\documentclass[reqno]{amsart}

\usepackage[utf8]{inputenc}
\usepackage{graphicx}
\usepackage{mathtools}
\usepackage{xcolor}
\usepackage[parfill]{parskip}
\usepackage{tikz}
\usetikzlibrary{arrows, backgrounds, calc, chains, decorations, patterns, positioning, shapes, decorations.pathreplacing}

\usepackage{amsfonts}
\usepackage{amsmath}
\usepackage[alphabetic]{amsrefs}
\usepackage{amssymb}
\usepackage{amstext}
\usepackage{amsthm}

\usepackage{thmtools}
\usepackage{caption}
\usepackage{enumitem}
\usepackage[centering]{geometry}
\usepackage{hyperref}
\usepackage{cleveref}
\usepackage{subcaption}
\usepackage{xspace}

\renewcommand{\th}{\textsuperscript{th}\xspace}
\newcommand{\st}{\textsuperscript{st}\xspace}

\newcommand{\rd}{\textsuperscript{rd}\xspace}

\newcommand{\area}{\mathsf{area}}
\newcommand{\dinv}{\mathsf{dinv}}
\newcommand{\tdinv}{\mathsf{tdinv}}
\newcommand{\cdinv}{\mathsf{cdinv}}
\newcommand{\falldinv}{\mathsf{falldinv}}

\newcommand{\shift}{\mathsf{shift}}

\newcommand{\RP}{\mathsf{RP}} %
\newcommand{\LRD}{\mathsf{LRD}} %
\newcommand{\LRP}{\mathsf{LRP}} %
\newcommand{\FL}{\mathsf{FL}} %

\newcommand{\sgn}{\textnormal{sgn}}

\newcommand{\verticalsteps}{\mathcal{V}}
\newcommand{\horizontalsteps}{\mathcal{H}}
\newcommand{\decoratedsteps}{\mathcal{D}}
\newcommand{\bigsteps}{\mathcal{B}}
\newcommand{\southsteps}{\mathcal{S}}

\colorlet{standardcolor}{blue!75}
\definecolor{decoratedcolor}{HTML}{b80026}
\definecolor{bigcolor}{HTML}{ed0031}
\colorlet{southcolor}{black!40}
\colorlet{gridcolor}{gray!40}
\colorlet{areacolor}{gray!15}
\colorlet{diagonalcolor}{gray!60}

\newcommand{\Ht}{\widetilde{H}}

\newcommand{\Z}{\mathbb{Z}}

\def\multiset#1#2{\ensuremath{\left(\kern-.3em\left(\genfrac{}{}{0pt}{}{#1}{#2}\right)\kern-.3em\right)}}

\newcommand{\<}{\langle}
\renewcommand{\>}{\rangle}

\theoremstyle{plain}

\theoremstyle{definition}
\newtheorem{theorem}{Theorem}[section]
\newtheorem{conjecture}[theorem]{Conjecture}

\newtheorem{definition}[theorem]{Definition}
\newtheorem{lemma}[theorem]{Lemma}
\newtheorem{proposition}[theorem]{Proposition}

\theoremstyle{remark}
\newtheorem{remark}[theorem]{Remark}

\makeatletter%
\DeclareRobustCommand*{\bfseries}{%
	\not@math@alphabet\bfseries\mathbf
	\fontseries\bfdefault\selectfont
	\boldmath
}
\makeatother

\makeatletter
\def\paragraph{\@startsection{paragraph}{4}{\z@}%
                                    {3.25ex \@plus1ex \@minus.2ex}%
                                    {-1em}%
                                    {\normalfont\normalsize\bfseries}}

\def\subparagraph{\@startsection{subparagraph}{5}{\parindent}%
                                       {3.25ex \@plus1ex \@minus .2ex}%
                                       {-1em}%
                                      {\normalfont\normalsize\bfseries}}
\makeatother

\makeatletter

\pgfkeys{
	/tikz/sharp angle/.code={%
		\pgfsetarrowoptions{sharp >}{#1}%
		\pgfsetarrowoptions{sharp <}{-#1}%
	},
	/tikz/sharp > angle/.code={%
		\pgfsetarrowoptions{sharp >}{#1}%
	},
	/tikz/sharp < angle/.code={%
		\pgfsetarrowoptions{sharp <}{#1}%
	},
	/tikz/sharp protrude/.code=\csname if#1\endcsname\qrr@tikz@sharp@z@-0.05\p@\else\qrr@tikz@sharp@z@\z@\fi,
	/tikz/sharp protrude/.default=true
}

\newdimen\qrr@tikz@sharp@z@
\qrr@tikz@sharp@z@\z@
\pgfarrowsdeclare{sharp >}{sharp >}{%
	\edef\pgf@marshal{\noexpand\pgfutil@in@{and}{\pgfgetarrowoptions{sharp >}}}%
	\pgf@marshal
	\ifpgfutil@in@
	\edef\pgf@tempa{\pgfgetarrowoptions{sharp >}}
	\expandafter\qrr@tikz@sharp@parse\pgf@tempa\@qrr@tikz@sharp@parse
	\else
	\qrr@tikz@sharp@parse\pgfgetarrowoptions{sharp >}and-\pgfgetarrowoptions{sharp >}\@qrr@tikz@sharp@parse
	\fi
	\pgfmathparse{max(\pgf@tempa,\pgf@tempb,0)}%
	\let\qrr@tikz@sharp@max\pgfmathresult
	\pgfmathsetlength\pgf@xa{.5*\pgflinewidth * tan(\qrr@tikz@sharp@max)}%
	\pgfarrowsleftextend{+\pgf@xa}%
	\pgfarrowsrightextend{+\pgf@xa}%
}{%
	\edef\pgf@marshal{\noexpand\pgfutil@in@{and}{\pgfgetarrowoptions{sharp >}}}%
	\pgf@marshal
	\ifpgfutil@in@
	\edef\pgf@tempa{\pgfgetarrowoptions{sharp >}}
	\expandafter\qrr@tikz@sharp@parse\pgf@tempa\@qrr@tikz@sharp@parse
	\else
	\qrr@tikz@sharp@parse\pgfgetarrowoptions{sharp >}and-\pgfgetarrowoptions{sharp >}\@qrr@tikz@sharp@parse
	\fi
	\pgfmathsetlength\pgf@ya{.5*\pgflinewidth * tan(max(\pgf@tempa,\pgf@tempb,0))}%
	\pgfmathsetlength\pgf@xa{-.5*\pgflinewidth * tan(\pgf@tempa)}%
	\pgfmathsetlength\pgf@xb{-.5*\pgflinewidth * tan(\pgf@tempb)}%
	\advance\pgf@xa\pgf@ya
	\advance\pgf@xb\pgf@ya
	\ifdim\pgf@xa>\pgf@xb
	\pgftransformyscale{-1}%
	\pgf@xc\pgf@xb
	\pgf@xb\pgf@xa
	\pgf@xa\pgf@xc
	\fi
	\pgfpathmoveto{\pgfqpoint{\qrr@tikz@sharp@z@}{.5\pgflinewidth}}%
	\pgfpathlineto{\pgfqpoint{\pgf@xa}{.5\pgflinewidth}}%
	\pgfpathlineto{\pgfqpoint{\pgf@ya}{+0pt}}%
	\pgfpathlineto{\pgfqpoint{\pgf@xb}{-.5\pgflinewidth}}%
	\pgfpathlineto{\pgfqpoint{\qrr@tikz@sharp@z@}{-.5\pgflinewidth}}%
	\pgfusepathqfill
}
\pgfarrowsdeclare{sharp <}{sharp <}{%
	\edef\pgf@marshal{\noexpand\pgfutil@in@{and}{\pgfgetarrowoptions{sharp <}}}%
	\pgf@marshal
	\ifpgfutil@in@
	\edef\pgf@tempa{\pgfgetarrowoptions{sharp <}}
	\expandafter\qrr@tikz@sharp@parse\pgf@tempa\@qrr@tikz@sharp@parse
	\else
	\expandafter\qrr@tikz@sharp@parse\pgfgetarrowoptions{sharp <}and-\pgfgetarrowoptions{sharp <}\@qrr@tikz@sharp@parse
	\fi
	\pgfmathparse{max(\pgf@tempa,\pgf@tempb,0)}%
	\let\qrr@tikz@sharp@max\pgfmathresult
	\pgfmathsetlength\pgf@xa{.5*\pgflinewidth * tan(\qrr@tikz@sharp@max)}%
	\pgfarrowsleftextend{+\pgf@xa}%
	\pgfarrowsrightextend{+\pgf@xa}%
}{%
	\edef\pgf@marshal{\noexpand\pgfutil@in@{and}{\pgfgetarrowoptions{sharp <}}}%
	\pgf@marshal
	\ifpgfutil@in@
	\edef\pgf@tempa{\pgfgetarrowoptions{sharp <}}
	\expandafter\qrr@tikz@sharp@parse\pgf@tempa\@qrr@tikz@sharp@parse
	\else
	\expandafter\qrr@tikz@sharp@parse\pgfgetarrowoptions{sharp <}and-\pgfgetarrowoptions{sharp <}\@qrr@tikz@sharp@parse
	\fi
	\pgfmathsetlength\pgf@ya{.5*\pgflinewidth * tan(max(\pgf@tempa,\pgf@tempb,0))}%
	\pgfmathsetlength\pgf@xa{-.5*\pgflinewidth * tan(\pgf@tempa)}%
	\pgfmathsetlength\pgf@xb{-.5*\pgflinewidth * tan(\pgf@tempb)}%
	\advance\pgf@xa\pgf@ya
	\advance\pgf@xb\pgf@ya
	\ifdim\pgf@xa>\pgf@xb
	\pgftransformyscale{-1}%
	\pgf@xc\pgf@xb
	\pgf@xb\pgf@xa
	\pgf@xa\pgf@xc
	\fi
	\pgfpathmoveto{\pgfqpoint{\qrr@tikz@sharp@z@}{.5\pgflinewidth}}%
	\pgfpathlineto{\pgfqpoint{\pgf@xa}{.5\pgflinewidth}}%
	\pgfpathlineto{\pgfqpoint{\pgf@ya}{+0pt}}%
	\pgfpathlineto{\pgfqpoint{\pgf@xb}{-.5\pgflinewidth}}%
	\pgfpathlineto{\pgfqpoint{\qrr@tikz@sharp@z@}{-.5\pgflinewidth}}%
	\pgfusepathqfill
}
\def\qrr@tikz@sharp@parse#1and#2\@qrr@tikz@sharp@parse{\def\pgf@tempa{#1}\def\pgf@tempb{#2}}

\makeatother

\title[Rectangular paths and Delta conjectures]{Falling stars: a fall-decorated rational shuffle theorem}
\author{Alessandro Iraci \and Roberto Pagaria \and Giovanni Paolini}

\address{Alessandro Iraci \newline \textup{Università Telematica Pegaso, Dipartimento di Ingegneria}\newline Centro Direzionale Isola F2 - Napoli\\ Italy.}
\email{alessandro.iraci@unipegaso.it}

\address{Roberto Pagaria \newline \textup{Università di Bologna, Dipartimento di Matematica}\newline Piazza di Porta San Donato 5 - 40126 Bologna\\ Italy.}
\email{roberto.pagaria@unibo.it}

\address{Giovanni Paolini \newline \textup{Università di Bologna, Dipartimento di Matematica}\newline Piazza di Porta San Donato 5 - 40126 Bologna\\ Italy.}
\email{g.paolini@unibo.it}

\begin{document}

\begin{abstract}
    In this paper, we formulate a rational analog of the fall Delta theorem and the Delta square conjecture.
    We find a new dinv statistic on fall-decorated paths on a $(m+k) \times (n+k)$ rectangle that simultaneously extends the previously known dinv statistics on decorated square objects and non-decorated rectangular objects.
    We prove a symmetric function formula for the $q,t$-generating function of fall-decorated rectangular Dyck paths as a skewing operator applied to $e_{m,n+km}$ and, conditionally on the rectangular paths conjecture, an analogous formula for fall-decorated rectangular paths.
\end{abstract}

\maketitle

\section{Introduction}

In the 1990s, Garsia and Haiman set out to prove the Schur positivity of the (modified) Macdonald polynomials by showing them to be the bi-graded Frobenius characteristic of certain Garsia-Haiman modules \cite{GarsiaHaiman1993GradedRepresentationModel}. Their prediction was confirmed in 2001, when Haiman used the algebraic geometry of the Hilbert scheme to prove that the dimension of their modules is equal to $n!$ \cite{Haiman2001nFactorial}, thus proving the $n!$ theorem. In the course of these developments, it became clear that there were remarkable connections to be found between Macdonald polynomial theory and the representation theory of the symmetric group. For example, during their quest for Macdonald positivity, Garsia and Haiman introduced the $\mathfrak{S}_n$-module of \emph{diagonal harmonics}, i.e.\ the coinvariants of the diagonal action of $\mathfrak{S}_n$ on polynomials in two sets of $n$ variables, and they conjectured that its Frobenius characteristic is given by $\nabla e_n$, where $\nabla$ is the \emph{nabla} operator on symmetric functions introduced in \cite{BergeronGarsiaHaimanTesler1999IdentitiesPositivityConjectures}, which acts diagonally on Macdonald polynomials. Haiman proved this conjecture in 2002 \cite{Haiman2002HilbertScheme}.

The combinatorial side of things solidified when Haglund, Haiman, Loehr, Remmel, and Ulyanov then formulated the so-called \emph{shuffle conjecture}  \cite{HaglundHaimanLoehrRemmelUlyanov2005ShuffleConjecture}, i.e.\ they predicted a combinatorial formula for $\nabla e_n$ in terms of labeled Dyck paths, which are lattice paths using North and East steps going from $(0,0)$ to $(n,n)$ and staying weakly above the line connecting these two points (called the \emph{main diagonal}). Several years later, Haglund, Morse, and Zabrocki conjectured a \emph{compositional} refinement of the shuffle conjecture, which also specified all the points where the Dyck paths return to the main diagonal \cite{HaglundMorseZabrocki2012CompositionalShuffleConjecture}. This was the statement later proved by Carlsson and Mellit in \cite{CarlssonMellit2018ShuffleConjecture}, which implies the \emph{shuffle theorem}.

Over the years, this subject has revealed itself to be extremely fruitful and to have striking connections to other fields of mathematics including elliptic Hall algebras \cites{SchiffmannVasserot2011EllipticHallAlgebra,BlasiakHaimanMorsePunSeelinger2023ShuffleAnyLine}, affine Hecke algebras \cite{CarlssonMellit2018ShuffleConjecture}, Springer fibers \cite{Mellit2020SpringerFibers}, the homology of torus knots \cites{GorskyNegut2015RefinedKnotInvariants,Mellit2022HomologyTorusKnots}, and the shuffle algebra of symmetric functions \cite{Negut2014ShuffleAlgebra}.

In this paper, we add a few (conjectural) formulas to the substantial list of variants and generalizations inspired by the success story of the shuffle theorem; that is, equations with a symmetric function related to Macdonald polynomials on one side and lattice paths combinatorics on the other. Furthermore, we support one of these conjectures by proving a non-trivial special case.

One of the first shuffle-like formulas was conjectured in 2007 by Loehr and Warrington \cite{LoehrWarrington2007Square}. They predicted an expression of $\nabla \omega(p_n)$ in terms of \emph{square paths}, i.e.\ lattice paths from $(0,0)$ to $(n,n)$ using only North and East steps and ending with an East step (without the restriction of staying above the main diagonal). Their formula was proved by Sergel in \cite{Sergel2017SquarePaths} to be a consequence of the shuffle theorem.

Next, Haglund, Remmel, and Wilson formulated the \emph{Delta conjecture} \cite{HaglundRemmelWilson2018DeltaConjecture}, a pair of conjectures for the symmetric function $\Delta'_{e_{n-k-1}}e_n$ in terms of decorated Dyck paths, where $k$ decorations are placed on either \emph{rises} or \emph{valleys} of the path. The symmetric function operator $\Delta'_f$ acts diagonally on the Macdonald polynomials and generalizes $\nabla$, in a sense.
The rise version of the Delta conjecture was proved by D'Adderio and Mellit in \cite{DAdderioMellit2022CompositionalDelta}, using the compositional refinement in \cite{DAdderioIraciVandenWyngaerd2021ThetaOperators}.
A \emph{Delta square conjecture} was stated in \cite{DAdderioIraciVandenWyngaerd2019DeltaSquareConjecture} and is still open today; it extends (the rise version of) the Delta conjecture in the same way as the square paths theorem extends the shuffle theorem.
The valley version also has similar extensions \cites{QiuWilson2020ValleyExtendedDelta,IraciVandenWyngaerd2021ValleySquare}, but it lacks a compositional version and is still open.

Around the same time as the formulation of the Delta conjecture, the story was extended to rectangular Dyck paths: paths from $(0,0)$ to $(m,n)$ staying above the main diagonal. In \cite{BergeronGarsiaSergelLevenXin2016PlethysticOperators}, building on the work in \cite{GorskyNegut2015RefinedKnotInvariants}, Bergeron, Garsia, Sergel, and Xin conjectured that a certain symmetric function related to the elliptic Hall algebra studied by Schiffmann and Vasserot \cite{SchiffmannVasserot2011EllipticHallAlgebra} can be expressed in terms of rectangular Dyck paths. Their prediction was recently proved by Mellit \cite{Mellit2021Rational}.

A first attempt to unify these last two shuffle-like conjectures appears in \cite{IraciPagariaPaoliniVandenWyngaerd2023}, where the authors of the present paper and Vanden Wyngaerd propose a univariate rational rise Delta conjecture, in terms of Dyck paths on a rectangle that lie above a certain ``broken'' diagonal. In the present work, we perfect the attempt, giving a complete formulation in terms of fall-decorated Dyck paths instead, which we manage to prove using the techniques introduced by Gillespie, Gorsky, and Griffin in \cite{GillespieGorskyGriffin2025}.

Our main result is a rational shuffle theorem for rectangular paths, which generalizes both the rational shuffle theorem for Dyck paths in \cite{Mellit2021Rational} (cf.\ \cite{BlasiakHaimanMorsePunSeelinger2023ShuffleAnyLine}) and the rise Delta theorem in \cites{HaglundRemmelWilson2018DeltaConjecture, DAdderioMellit2022CompositionalDelta}.

\begin{theorem}[Fall-decorated rational shuffle theorem]
    \label{thm:main}
    For any $m, n, k \in \mathbb{N}$ with $m > 0$, we have
    \begin{equation}
        s^\perp_{(m-1)^k} e_{m,n+km} = \sum_{\pi \in \LRD(m+k,n+k)_{\ast k}} q^{\dinv(\pi)} t^{\area(\pi)} x^\pi.
    \end{equation}    
\end{theorem}

For the definition of the symmetric function $e_{m,n}$, see~\Cref{def:hall-algebra-op}. The same approach works for paths above any ``broken'' line (see \Cref{def:broken-diagonal} and \Cref{rmk:any-slope} for details), but we will restrict to paths from $(0,0)$ to $(m+k, n+k)$ for simplicity. Conditionally on \cite{IraciPagariaPaoliniVandenWyngaerd2023}*{Conjecture~4.2}, we also prove the following.

\begin{theorem}[Fall-decorated rectangular shuffle theorem]
    \label{thm:rectangular}
    If \cite{IraciPagariaPaoliniVandenWyngaerd2023}*{Conjecture~4.2} holds, then for any $m, n, k \in \mathbb{N}$ and $d = \gcd(m,n)$, we have
    \begin{equation}
        s^\perp_{(m-1)^k} \frac{[m]_q}{[d]_q} p_{m,n+km} = \sum_{\pi \in \LRP(m+k,n+k)_{\ast k}} q^{\dinv(\pi)} t^{\area(\pi)} x^\pi.
    \end{equation}
    In particular, the identity holds if $d = 1$.
\end{theorem}

The paper is structured as follows. In \Cref{sec:symmetric-functions}, we introduce the notation for the symmetric functions we use, and in \Cref{sec:combinatorial-definitions}, we do the same for the combinatorial objects.

\Cref{sec:combinatorial-arguments} and \Cref{sec:dinv} are set up for the proof of \Cref{thm:main}, which will be completed in \Cref{sec:proof-main}. In \Cref{sec:combinatorial-arguments}, we adapt the main arguments in \cite{GillespieGorskyGriffin2025} to our case. First, in Subsection~\ref{ssc:skewing}, we show that applying $s^\perp_{(m-1)^k}$ to $e_{m,n+km}$ gives a signed sum of a certain subset of labeled rectangular paths of size $m \times (n+km)$.
Then, in Subsection~\ref{ssc:bijection} we describe a bijection between a certain subset of labeled rectangular paths of size $m \times (n+km)$ and all fall-labeled rectangular paths of size $(m+k) \times (n+k)$ with $k$ decorated falls, without any restriction on being above the main diagonal. This bijection makes use of a new representation for decorated objects, where decorated East steps are replaced by South steps; this representation preserves the area and the vertical distances of the vertical steps from the main diagonal.
Afterwards, in Subsection~\ref{ssc:involution} we exhibit a sign-reversing involution on our set with a unique fixed point for each element in $\LRP(m+k,n+k)_{\ast k}$.
Finally, in \Cref{sec:dinv} we show that the dinv of the fixed point is equal to the dinv of the corresponding element in $\LRP(m+k,n+k)_{\ast k}$, and in \Cref{sec:proof-main} we combine the pieces, completing the proof of \Cref{thm:main}.

We conclude the paper by highlighting some connections with the existing literature in \Cref{sec:links}, especially concerning the results about $D_\alpha$ operators by Blasiak et al.\ \cite{BlasiakHaimanMorsePunSeelinger2023ShuffleAnyLine} and the Theta operators by D'Adderio et al.\ \cite{DAdderioIraciVandenWyngaerd2021ThetaOperators}.

\section{Symmetric functions}
\label{sec:symmetric-functions}

For all undefined notation and unproven identities, we refer to \cite{DAdderioIraciVandenWyngaerd2022TheBible}*{Section~1}, where definitions, proofs, and/or references can be found. See also \cite{IraciPagariaPaoliniVandenWyngaerd2023}.

We denote by $\Lambda$ the graded algebra of symmetric functions with coefficients in $\mathbb{Q}(q,t)$, and by $\<\, , \>$ the \emph{Hall scalar product} on $\Lambda$, defined by declaring that the Schur functions form an orthonormal basis.

The standard bases of the symmetric functions that will appear in our calculations are the monomial $\{m_\lambda\}_{\lambda}$, complete $\{h_{\lambda}\}_{\lambda}$, elementary $\{e_{\lambda}\}_{\lambda}$, power $\{p_{\lambda}\}_{\lambda}$, and Schur $\{s_{\lambda}\}_{\lambda}$ bases.

For $f \in \Lambda$, we denote by $f^\perp$ the operator that is the adjoint of the product by $f$ with respect to the Hall scalar product, that is, for every $g, h \in \Lambda$, we have $\< f^\perp g, h \> = \< g, fh \>$.

For a partition $\mu \vdash n$, we denote by \[ \Ht_\mu \coloneqq \Ht_\mu[X] = \Ht_\mu[X; q,t] = \sum_{\lambda \vdash n} \widetilde{K}_{\lambda \mu}(q,t) s_{\lambda} \] the \emph{(modified) Macdonald polynomials}, where \[ \widetilde{K}_{\lambda \mu} \coloneqq \widetilde{K}_{\lambda \mu}(q,t) = K_{\lambda \mu}(q,1/t) t^{n(\mu)} \] are the \emph{(modified) Kostka coefficients} (see \cite{Haglund2008Book}*{Chapter~2} for more details).

Macdonald polynomials form a basis of the algebra of symmetric functions $\Lambda$. This is a modification of the basis introduced by Macdonald \cite{Macdonald1995Book}.

If we identify the partition $\mu$ with its Young diagram, i.e.\ with the collection of cells $\{(i,j)\mid 1\leq i\leq \mu_j, 1\leq j\leq \ell(\mu)\}$, then for each cell $c\in \mu$ we refer to the \emph{arm}, \emph{leg}, \emph{co-arm} and \emph{co-leg} (denoted respectively by $a_\mu(c), l_\mu(c), a_\mu'(c), l_\mu'(c)$) as the number of cells in $\mu$ that are strictly to the right, below, to the left and above $c$ in $\mu$, respectively (see Figure~\ref{fig:notation}).

\begin{figure}
    \centering
    \begin{tikzpicture}[scale=0.4]
        \draw[gray,opacity=.6](0,0) grid (15,10);
        \fill[white] (1,-0.1)|-(3,1) |- (6,3) |- (9,7) |- (15.1,8) |- (1,-.1);
        \fill[blue, opacity=.15] (0,7) rectangle (9,8) (3,3) rectangle (4,10);
        \fill[blue, opacity=.5] (3,7) rectangle (4,8);
        \draw (6,7.5) node {\tiny{Arm}} (3.5,5) node[rotate=90] {\tiny{Leg}} (3.5, 9) node[rotate = 90] {\tiny{Co-leg}} (1.5,7.5) node {\tiny{Co-arm}} ;
    \end{tikzpicture}
    \caption{Arm, leg, co-arm, and co-leg of a cell of a partition.}
    \label{fig:notation}
\end{figure}

For every partition $\mu$, we define the following constants:

\[ B_{\mu} \coloneqq B_{\mu}(q,t) = \sum_{c \in \mu} q^{a_{\mu}'(c)} t^{l_{\mu}'(c)}, \qquad
    \Pi_{\mu} \coloneqq \Pi_{\mu}(q,t) = \prod_{c \in \mu / (1)} (1-q^{a_{\mu}'(c)} t^{l_{\mu}'(c)}). \]

Notice that in the definition of $\Pi_\mu$, the exponents $a_{\mu}'$ and $l_{\mu}'$ are evaluated with respect to the shape $\mu$, not $\mu / (1)$, the latter only denoting that we should skip the cell $(1,1)$ in the product (otherwise the product would be $0$).

We will make extensive use of the \emph{plethystic notation} (cf. \cite{Haglund2008Book}*{Chapter~1}). We also need several linear operators on $\Lambda$.

\begin{definition}[{\cite{BergeronGarsia1999ScienceFiction}*{3.11}}]
    \label{def:nabla}
    We define the linear operator $\nabla \colon \Lambda \rightarrow \Lambda$ on the eigenbasis of Macdonald polynomials as \[ \nabla \Ht_\mu = e_{\lvert \mu \rvert}[B_\mu] \Ht_\mu. \]
\end{definition}

\begin{definition}
    \label{def:pi}
    We define the linear operator $\mathbf{\Pi} \colon \Lambda \rightarrow \Lambda$ on the eigenbasis of Macdonald polynomials as \[ \mathbf{\Pi} \Ht_\mu = \Pi_\mu \Ht_\mu \] where we conventionally set $\Pi_{\varnothing} \coloneqq 1$.
\end{definition}

\begin{definition}
    \label{def:delta}
    For $f \in \Lambda$, we define the linear operators $\Delta_f, \Delta'_f \colon \Lambda \rightarrow \Lambda$ on the eigenbasis of Macdonald polynomials as \[ \Delta_f \Ht_\mu = f[B_\mu] \Ht_\mu, \qquad \qquad \Delta'_f \Ht_\mu = f[B_\mu-1] \Ht_\mu. \]
\end{definition}

Observe that on the vector space of homogeneous symmetric functions of degree $n$, denoted by $\Lambda^{(n)}$, the operator $\nabla$ equals $\Delta_{e_n}$.

From now on, let $M \coloneqq (1-q)(1-t)$.

\begin{definition}[{\cite{DAdderioIraciVandenWyngaerd2021ThetaOperators}*{(28)}}]
    \label{def:theta}
    For any symmetric function $f \in \Lambda^{(n)}$ we define the \emph{Theta operators} on $\Lambda$ in the following way: for every $F \in \Lambda^{(m)}$ we set
    \begin{equation*}
        \Theta_f F  \coloneqq
        \left\{\begin{array}{ll}
            0                                                           & \text{if } n \geq 1 \text{ and } m=0 \\
            f \cdot F                                                   & \text{if } n=0 \text{ and } m=0      \\
            \mathbf{\Pi} f \left[\frac{X}{M}\right] \mathbf{\Pi}^{-1} F & \text{otherwise}
        \end{array}
        \right. ,
    \end{equation*}
    and we extend by linearity the definition to any $f, F \in \Lambda$.
\end{definition}

It is clear that $\Theta_f$ is linear. In addition, if $f$ is homogeneous of degree $k$, then so is $\Theta_f$:
\[\Theta_f \Lambda^{(n)} \subseteq \Lambda^{(n+k)} \qquad \text{ for } f \in \Lambda^{(k)}. \]

Finally, we need to refer to \cite{BergeronGarsiaSergelLevenXin2015ShuffleConjectures}*{Algorithm~4.1} (see also \cite{BergeronGarsiaSergelLevenXin2016PlethysticOperators}*{Definition~1.1, Theorem~2.5}).

\begin{definition}
    Let $m, n > 0$. Let $a,b,c,d \in \mathbb{N}$ be such that $a+c=m$, $b+d=n$, $ad-bc = \gcd(m,n)$. We recursively define $Q_{m,n}$ as an operator on $\Lambda$ by \[ Q_{m,n} = \frac{1}{M} \left( Q_{c,d} Q_{a,b} - Q_{a,b} Q_{c,d} \right), \] with the base cases \[ Q_{1,0} = D_0 = \mathsf{id} - M \Delta_{e_1} \quad \text{ and } \quad Q_{0,1} = - \underline{e_1} \] (where $\underline{f}$ is the multiplication by $f$).
\end{definition}

\begin{definition}
    \label{def:hall-algebra-op}
    For a coprime pair $(a,b)$ and $f \in \Lambda^{(d)}$, we define $F_{a,b}(f)$ as follows. Let \[ f = \sum_{\lambda \vdash d} c_\lambda(q,t) \left( \frac{qt}{qt-1} \right)^{\ell(\lambda)} h_\lambda \left[ \frac{1-qt}{qt} X \right]. \] Then, we define \[ F_{a,b}(f) \coloneqq \sum_{\lambda \vdash d} c_\lambda(q,t) \prod_{i=1}^{\ell(\lambda)} Q_{\lambda_i a, \lambda_i b}(1). \]
\end{definition}

In the literature, $F_{a,b}(f)$ is often written as $f[-MX^{a,b}]$.
We use the notation \[ e_{m,n} \coloneqq F_{a,b}(e_d), \qquad p_{m,n} \coloneqq F_{a,b}(p_d) \] where $m = ad, n = bd$, and $\gcd(a,b) = 1$. Beware: $e_{4,2} = F_{2,1}(e_2)$, but $e_{42} = e_4 e_2$. Notice that, in the literature, we sometimes see $e_{m,n}$ denoted as $E_{m,n} \cdot 1$, where $E$ is the elliptic Hall algebra operator.

\section{Combinatorial definitions}
\label{sec:combinatorial-definitions}

The objects we are concerned with are \emph{rectangular Dyck paths} and \emph{rectangular paths}. The following definitions first appeared in \cite{BergeronGarsiaSergelLevenXin2016PlethysticOperators} for rectangular Dyck paths and in \cite{IraciPagariaPaoliniVandenWyngaerd2023} for rectangular paths.

\subsection{Rectangular paths}

\begin{definition}
    A \emph{rectangular path} of size $m \times n$ is a lattice path composed of unit vertical and horizontal steps, going from $(0,0)$ to $(m,n)$, and ending with a horizontal step. A \emph{rectangular Dyck path} is a rectangular path that lies weakly above the diagonal $my = nx$ (called the \emph{main diagonal}).
\end{definition}

\begin{figure}
    \centering
    \begin{tikzpicture}[scale=0.6]
        \draw[gridcolor] (0,0) grid (7,9);
        \draw[dashed, diagonalcolor] (0, -11/7) -- (7, 9 - 11/7);
        \draw[diagonalcolor] (0,0) -- (7,9);

        \draw[standardcolor, very thick] (0,0) -- (0,1) -- (1,1) -- (2,1) -- (2,2) -- (2,3) -- (2,4) -- (3,4) -- (4,4) -- (4,5) -- (4,6) -- (4,7) -- (5,7) -- (6,7) -- (6,8) -- (6,9) -- (7,9);
    \end{tikzpicture}
    \caption{A $7 \times 9$ rectangular path with its base diagonal (dashed) and main diagonal (solid).}
    \label{fig:rectangular-path}
\end{figure}

The requirement that the path end with a horizontal step is imposed by analogy with the original definition of a labeled square path \cite{LoehrWarrington2007Square} and also because the resulting combinatorics better matches the algebraic expressions. We denote the sets of rectangular paths of size $m \times n$ as $\RP(m,n)$.

Given a rectangular path $\pi$, denote by $\horizontalsteps = \horizontalsteps(\pi)$ (resp.\ $\verticalsteps = \verticalsteps(\pi)$) the set of horizontal (resp.\ vertical) steps of~$\pi$. The set $\horizontalsteps \sqcup \verticalsteps$ is totally ordered by traversing the path from $(0, 0)$ to $(m, n)$: we write $i < j$ if step $i$ precedes step $j$.
If $i$ is a step of $\pi$ and $p \in \Z$ is an integer, we write $i + p$ to denote the step that lies $p$ positions away from $i$ along $\pi$; thus, $i+1$ is the step immediately after~$i$ (if it exists).

\begin{definition}
    For a $m \times n$ rectangular path $\pi$ and a step $i \in \horizontalsteps$ (resp.\ $i \in \verticalsteps$), denote by $v_i = v_i(\pi)$ the (signed) vertical distance between the right-most (resp.\ bottom-most) point of $i$ and the main diagonal;
    the sign is positive if the endpoint is above the main diagonal.
    Define the \emph{vertical area word} of the path as the sequence $(v_i)_{i \in \horizontalsteps}$, i.e., the sequence of heights of the horizontal steps of the path. Set $s \coloneqq - \min \{v_i \mid i \in \horizontalsteps\}$, which we call the \emph{shift} of the path. Note that $s = 0$ if $\pi$ is a rectangular Dyck path, and $s > 0$ otherwise.
\end{definition}

\begin{definition}
    The diagonal $m(y+s) = nx$ is called the \emph{base diagonal}. It is the lowest diagonal that intersects the path.
\end{definition}

\begin{definition}
    The \emph{area} of a rectangular path $\pi$ is $\area(\pi) \coloneqq \sum_{i\in \horizontalsteps} \lfloor v_i + s \rfloor$. This is the number of whole squares that lie entirely between the path $\pi$ and its base diagonal (including the squares below the line $y=0$, but excluding the squares to the right of the line $x=m$).
\end{definition}

\begin{remark}
    This definition differs from \cite{IraciPagariaPaoliniVandenWyngaerd2023}*{Definition~3.4} because we count the squares vertically below the path instead of the squares horizontally to the right of the path. This is irrelevant for now, as the two triangles outside of the $m \times n$ rectangle we are counting in the two cases are congruent and translated by an integer vector. However, it will become important when we define the area of decorated paths.
\end{remark}

For example, the path in Figure~\ref{fig:rectangular-path} has vertical area word
\[ \textstyle \left(-\frac{2}{7}, -\frac{11}{7}, \frac{1}{7}, -\frac{8}{7}, \frac{4}{7}, -\frac{5}{7}, 0\right). \]
Thus, its shift is $\frac{11}{7}$ and its area is $1 + 0 + 1 + 0 + 2 + 0 + 1 = 5$.

\subsection{Decorated rectangular paths}

In a similar fashion to the rise version of the Delta conjecture \cite{HaglundRemmelWilson2018DeltaConjecture} (which is now a theorem \cites{DAdderioMellit2022CompositionalDelta,BlasiakHaimanMorsePunSeelinger2023ProofExtendedDelta}), we introduce the concept of \emph{decorated falls} for rectangular paths.

\begin{definition}
    The \emph{falls} of a rectangular path are the horizontal steps that are immediately followed by another horizontal step. A \emph{fall-decorated rectangular path} is a pair $(\pi, \decoratedsteps)$ where $\pi$ is a rectangular path and $\decoratedsteps$ is a subset of its falls.
    Normally, we will simply write \emph{decorated} in place of \emph{fall-decorated}.
    For a decorated rectangular path, we denote by $\horizontalsteps$ the set of non-decorated horizontal steps, so that the set of all steps of $\pi$ is the disjoint union $\horizontalsteps \sqcup \verticalsteps \sqcup \decoratedsteps$ (totally ordered by traversing the path).
\end{definition}

\begin{definition}
    \label{def:broken-diagonal}
    For a decorated rectangular path of size $(m+k) \times (n+k)$ with $k$ decorated falls, we define the \emph{broken diagonal} to be the broken line segment from $(0, 0)$ to $(m+k,n+k)$ that proceeds with slope $\frac{n}{m}$ in columns containing non-decorated horizontal steps, and with slope $1$ in columns containing decorated falls.
\end{definition}

This definition is analogous to \cite{IraciPagariaPaoliniVandenWyngaerd2023}*{Definition~3.6} when reflecting the path with respect to the antidiagonal.
Note that if the path has no decorated falls, then the broken diagonal coincides with the main diagonal.

Since the broken diagonal never proceeds horizontally or vertically, it intersects any horizontal or vertical segment in at most one point.
We say that a point $(x, y)$ with $0 \leq x \leq m + k$ \textit{diagonally projects} onto a horizontal or vertical step if the vertical translation of the broken diagonal passing through $(x, y)$ intersects the step.
Formally, if the broken diagonal is denoted by $\mathrm{BD}$, there exists a unique $v \in \mathbb{R}$ such that $(x,y) \in \mathrm{BD}+ (0,v)$; the point $(x,y)$ diagonally projects onto a step $s$ if $s \cap (\mathrm{BD}+ (0,v)) \neq \varnothing$.

\begin{remark}
    \label{rmk:any-slope}
    The definition of broken diagonal extends to any slope, as follows. Fix any positive reals $a$ and $b$, and let $s = b/a$. Construct any broken line segment from $(0, -\{b\})$ to $(a + k, \lfloor b \rfloor + k)$ that proceeds with slope $1$ in $k$ of the columns and with slope $s$ in the remaining columns; here, $\{b\} \coloneqq b - \lfloor b \rfloor$ denotes the fractional part.
    Consider all fall-decorated rectangular paths from $(0, 0)$ to $(\lfloor a \rfloor + k, \lfloor b \rfloor + k)$ that lie above this line and have $k$ decorated falls exactly in the chosen columns of slope $1$.
    All results in this paper extend to these paths as well, but we will only treat the case where $a, b \in \mathbb{N}$ for simplicity.
\end{remark}

\begin{remark}
    If $m = n$, that is, the path is a square path, then there is a bijection between falls and rises (i.e.\ vertical steps preceded by vertical steps).
    This is not the case for rectangular paths. However, the reflection with respect to the antidiagonal defines a bijection between $(m+k) \times (n+k)$ rectangular Dyck paths with $k$ decorated falls and $(n+k) \times (m+k)$ rectangular Dyck paths with $k$ decorated rises, as in \cite{IraciPagariaPaoliniVandenWyngaerd2023}*{Definition~3.5}. This bijection will come into play in \Cref{sec:theta-operators}.
\end{remark}

\begin{definition}
    We define a \emph{decorated rectangular Dyck path} to be a decorated rectangular path that lies weakly above the broken diagonal.
\end{definition}

See Figure~\ref{fig:decorated-dyck} (left) for an example of such a path. In our figures, we use a $\ast$ to mark the decorated falls.

\begin{figure}
    \centering
    \begin{tikzpicture}[scale=.6, baseline={(gridcenter)}]
        \coordinate (gridcenter) at (4.5,3);  %
        
        \fill[areacolor] (0,1) rectangle (2,2);
        \fill[areacolor] (5,4) rectangle (6,5);
    
        \draw[step=1.0, gridcolor, thin] (0,0) grid (9,6);

        \begin{scope}
            \clip (0,0) rectangle (9,6);
            \draw[diagonalcolor] 
            (9,6) -- (8,5.5) -- (7,5) -- (6,4) -- (5,3.5) -- (4,3) -- (3,2) -- (2,1) -- (1,0.5) -- (0,0);
        \end{scope}

        \draw[standardcolor, very thick, -sharp >, sharp angle = -45] 
            (0, 0) -- (0, 2) -- (2, 2) -- (2, 4);

        \draw[decoratedcolor, very thick, sharp <-, sharp angle = -45]
            (2, 4) -- (4, 4);

        \draw[standardcolor, very thick, -sharp >, sharp angle = -45] 
            (4, 4) -- (5, 4) -- (5, 5) -- (6, 5) -- (6, 6);

        \draw[decoratedcolor, very thick, sharp <-, sharp angle = -45] 
            (6, 6) -- (7, 6);

        \draw[standardcolor, very thick] 
            (7, 6) -- (9, 6);

        \begin{scope}[decoratedcolor]
            \draw (6.5,6.3) node {$\ast$};
            \draw (3.5,4.3) node {$\ast$};
            \draw (2.5,4.3) node {$\ast$};
        \end{scope}
    \end{tikzpicture}
    \qquad\qquad
    \begin{tikzpicture}[scale=.6, baseline={(gridcenter)}]
        \coordinate (gridcenter) at (3,2);  %
    
        \fill[areacolor] (0,1) rectangle (2,2);
        \fill[areacolor] (3,2) rectangle (4,3);

        \draw[step=1.0, gridcolor, thin] (0,0) grid (6,4);

        \begin{scope}
            \clip (0,0) rectangle (6,3);
            \draw[diagonalcolor, thin] 
            (0, 0) -- (6, 3);
        \end{scope}

        \draw[standardcolor, very thick]
            (0, 0) -- (0,2) -- (1.95,2) -- (1.95, 4) -- (2, 4);

        \draw[decoratedcolor, very thick, -sharp >, sharp angle = 45]
            (2, 4) -- (2.05, 4) -- (2.05, 2);

        \draw[standardcolor, very thick, sharp <-, sharp angle = 45]
            (2.05, 2) -- (3, 2) -- (3, 3) -- (3.95, 3) -- (3.95, 4) -- (4, 4);

        \draw[decoratedcolor, very thick, -sharp >, sharp angle = 45]
            (4, 4) -- (4.05, 4) -- (4.05, 3);

        \draw[standardcolor, very thick, sharp <-, sharp angle = 45]
            (4.05, 3) -- (6, 3);

        \begin{scope}[decoratedcolor]
            \draw (2.3,3.5) node {$\ast$};
            \draw (2.3,2.5) node {$\ast$};
            \draw (4.3,3.5) node {$\ast$};
        \end{scope}
    \end{tikzpicture}

    \caption{On the left, a decorated rectangular Dyck path of size $(6 + 3) \times (3+3)$ with its broken diagonal. On the right, the ENS representation of the same path.
    Decorated steps are highlighted in dark red.
    The three squares contributing to the area are highlighted in gray.}
    \label{fig:decorated-dyck}
\end{figure}

The definitions of \emph{vertical distance} and \emph{vertical area word} extend to decorated paths as well, using the broken diagonal in place of the main diagonal.
The definition of \emph{area} also extends to decorated paths, where the sum over $i \in \horizontalsteps$ excludes the decorated steps.
For example, the area of the path in Figure~\ref{fig:decorated-dyck} is equal to $3$.

An alternative way to draw a $(m+k) \times (n+k)$ rectangular path $\pi$ with $k$ decorated falls is achieved by replacing decorated horizontal steps $i \in \decoratedsteps$ by decorated South steps.
The result is a lattice path from $(0, 0)$ to $(m, n)$ with $m$ East steps (indexed by $\horizontalsteps$), $n + k$ North steps (indexed by $\verticalsteps$), and $k$ decorated South steps (indexed by $\decoratedsteps$), where a sequence of consecutive South steps must be followed by an East step; see \Cref{fig:decorated-dyck} (right).
We call this alternative representation the \textit{ENS representation} of $\pi$. In this framework, the extension suggested in \Cref{rmk:any-slope} is even more apparent.

Conveniently, the vertical distance $v_i$ of a step from the broken diagonal (in the usual representation) coincides with the vertical distance of the corresponding step from the main diagonal in the ENS representation.
In addition, the area of a path $\pi$ is the number of whole squares between $\pi$ and its base diagonal in the ENS representation.

\subsection{Labeled paths}

Finally, we need to introduce labeled objects.

\begin{definition}
    A \emph{labeling} of a decorated rectangular path is an assignment of a positive integer label to each vertical step of the path such that consecutive vertical steps are assigned strictly increasing labels.
    A \emph{labeled decorated rectangular path} is a triple $(\pi, \decoratedsteps, w)$ where $(\pi, \decoratedsteps)$ is a decorated rectangular path and $w$ is a labeling.
\end{definition}

\begin{remark}
    \label{rmk:notation}
    With an abuse of notation, we will sometimes write $\pi$ to mean $(\pi, \decoratedsteps, w)$.
\end{remark}

We say that a labeling is \emph{standard} if the set of labels is $[n+k] \coloneqq \{1, \dots, n+k\}$, where $n+k$ is the height of the path, and we denote by $w_i$ the label assigned to the vertical step $i \in \verticalsteps$.

\begin{figure}
    \centering
    \begin{minipage}{.45\textwidth}
        \centering
        \begin{tikzpicture}[scale=0.6]
            \draw[draw=none, use as bounding box] (0,0) rectangle (9,8);
            \draw[gridcolor, thin] (0,0) grid (9,7);
            \draw[diagonalcolor] (0,0) -- (9,7);

            \draw[standardcolor, very thick] 
                (9,7) -- (8,7) -- (8,6) -- (8,5)  -- (7,5) -- (6,5) -- (5,5) -- (5,4) 
                -- (4,4) -- (3,4) -- (2,4) -- (2,3) -- (2,2) -- (1,2) -- (0,2) -- (0,1) -- (0,0);

            \begin{scope}[every node/.style = {font=\scriptsize}]
              \draw
                (-0.3,0.5) node {$1$}
                (-0.3,1.5) node {$2$}
                (1.7,2.5) node {$4$}
                (1.7,3.5) node {$7$}
                (4.7,4.5) node {$2$}
                (7.7,5.5) node {$4$}
                (7.7,6.5) node {$7$};
            \end{scope}
            
        \end{tikzpicture}
    \end{minipage}%
    \begin{minipage}{.45\textwidth}
        \centering
        \begin{tikzpicture}[scale=.6]
            \draw[draw=none, use as bounding box] (0,0) rectangle (9,8);
            \draw[gridcolor, thin] (0,0) grid (9,7);

            \begin{scope}
                \clip (0,0) rectangle (9,7);
                \draw[diagonalcolor]
                (9,7) -- (8, 6+1/3) -- (7,5+2/3) -- (6,4+2/3) -- (5,4) -- (4,3+1/3) -- (3,2+1/3) -- (2,1+1/3) -- (0,0);
            
            \end{scope}

            \begin{scope}[standardcolor, very thick]
                \draw (0, 0) -- (0, 3) -- (1, 3) -- (1, 5) -- (2, 5);

                \draw[-sharp >, sharp angle = -45] (4, 5) -- (5, 5) -- (5, 6) -- (6, 6) -- (6, 7);
                \draw (7, 7) -- (9, 7);
            \end{scope}

            \begin{scope}[decoratedcolor, very thick]
                \draw (2, 5) -- (4, 5);
                \draw[sharp <-, sharp angle = -45] (6, 7) -- (7, 7);
            \end{scope}

            \begin{scope}[decoratedcolor]
                \draw (6.5,7.3) node {$\ast$};
                \draw (3.5,5.3) node {$\ast$};
                \draw (2.5,5.3) node {$\ast$};
            \end{scope}

            \begin{scope}[every node/.style = {font=\scriptsize}]
              \draw
                (-0.3,0.5) node {$1$}
                (-0.3,1.5) node {$2$}
                (-0.3,2.5) node {$4$}
                (0.7,3.5) node {$2$}
                (0.7,4.5) node {$4$}
                (4.7,5.5) node {$7$}
                (5.7,6.5) node {$1$};
            \end{scope}
            
        \end{tikzpicture}
    \end{minipage}
    \caption{A $9 \times 7$ labeled rectangular path (left) and labeled decorated Dyck path (right).}
\end{figure}

We denote the sets of labeled rectangular paths and labeled rectangular Dyck paths of size $m \times n$ by $\LRP(m,n)$ and $\LRD(m,n)$ respectively, and the sets of labeled fall-decorated rectangular paths and labeled fall-decorated rectangular Dyck paths of size $(m+k) \times (n+k)$ with $k$ decorations by $\LRP(m+k,n+k)_{\ast k}$ and $\LRD(m+k,n+k)_{\ast k}$, respectively.

\begin{definition}
    Given a labeled decorated rectangular path $(\pi, \decoratedsteps, w)$, let $x^w = \prod_{i \in \verticalsteps} x_{w_i}$.
    We sometimes write $x^\pi$ in place of $x^w$.
\end{definition}

It will be useful later to consider labelings in the alphabet \[ \mathbb{Z}_+ \cup \overline{\mathbb{Z}}_+ = 1 < 2 < 3 < \dots < \overline{1} < \overline{2} < \overline{3} < \dots. \]
In this case, we set $x_{\overline{i}} = y_i$ in the plethystic alphabet $X + Y$.

We now extend the definition of \emph{dinv} given in \cite{IraciPagariaPaoliniVandenWyngaerd2023} (cf.\ \cites{BergeronGarsiaSergelLevenXin2016PlethysticOperators, HicksSergel2015SimplerFormulaNumber, Mellit2021Rational}) to any fall-decorated rectangular path.

\begin{definition}
    \label{def:attack_relation}
    Let $\pi$ be a $(m+k) \times (n+k)$ (decorated) rectangular path, and let $i,j \in \verticalsteps$. We say that $i$ \emph{attacks} $j$ in $\pi$ (or $(i,j)$ is an \emph{attack relation} for $\pi$), and write $i \to j$, if
    \[ (v_i, i) <_{\text{lex}} (v_j, j) <_{\text{lex}} (v_i + 1, i). \]
\end{definition}

At this point, we can define the temporary dinv of a path.

\begin{definition}
    We define the \emph{temporary dinv} of a labeled (decorated)  rectangular path $(\pi, w)$ as
    \[ \tdinv(\pi, w) \coloneqq \# \{ i, j \in \verticalsteps \mid w_i < w_j \text{ and } i \to j \}. \]
\end{definition}

Geometrically, the temporary dinv counts all pairs of vertical steps $(i,j)$ such that $w_i < w_j$ and the bottom-most endpoint of step $j$ diagonally projects onto step $i$, where we include the bottom-most endpoint of step $i$ if $i < j$ and the top-most endpoint of step $i$ if $i > j$.

In order to complete the definition of dinv, we need to introduce a correction term, which is the \emph{cdinv} of the path. This is a generalization of the \emph{cdinv} defined in \cite{HicksSergel2015SimplerFormulaNumber} for Dyck paths and later in \cite{IraciPagariaPaoliniVandenWyngaerd2023} for rectangular paths.

First, we introduce some notation.

\begin{definition}
    Let $(\pi, \decoratedsteps)$ be a $(m+k) \times (n+k)$ fall-decorated rectangular path with $k$ decorated falls. 
    For $j \in \horizontalsteps$, we define $r_j \coloneqq \max \{r \in \mathbb{N} \mid j-p \in \mathcal{D} \, \, \, \forall \, 1 \leq p \leq r\}$ as the number of decorated falls immediately preceding the horizontal step~$j$.
\end{definition}

We are now ready to define the \textit{dinv correction} (or \textit{cdinv}).
We give two different definitions, which we prove to be equivalent in \Cref{prop:def_of_cdinv_coincide}.

\begin{definition}
    \label{def:cdinv}
    Let $(\pi, \decoratedsteps)$ be a fall-decorated rectangular path, and let
    \begin{align*}
        D_+ & \coloneqq \left\{ (i, j) \in \verticalsteps \times  \horizontalsteps \mid j < i \text{ and } v_i < v_j \leq v_i + 1 - \textstyle\frac{n}{m} \right\} \\
        D_- & \coloneqq \left\{ (i, j) \in \verticalsteps \times \horizontalsteps \mid j < i \text{ and } v_j \leq v_i < v_j + \textstyle\frac{n}{m} - 1 \right\} \\
        D_+^\ast & \coloneqq \left\{ (i, j) \in \decoratedsteps\times \horizontalsteps \mid i < j \text{ and } v_i \leq v_j  < v_i + 1- \textstyle\frac{n}{m} \right\} \\
        D_-^\ast & \coloneqq \left\{ (i, j) \in \decoratedsteps \times \horizontalsteps \mid i < j \text{ and } v_j < v_i  \leq  v_j - 1+ \textstyle\frac{n}{m} \right\} \\
         B & \coloneqq \{ i \in \verticalsteps \mid v_i < 0 \} \\
         B^* &\coloneqq \{ i \in \decoratedsteps \mid v_i \leq 0 \}.
    \end{align*}
    We define the \emph{dinv correction} of $(\pi, \decoratedsteps)$ as \[ \cdinv(\pi, \decoratedsteps) \coloneqq \# D_+ - \# D_- + \# D_+^\ast -\# D_-^\ast +\# B -\# B^\ast . \]
\end{definition}

\begin{figure}
    \centering
    \begin{tikzpicture}
        \draw[standardcolor, very thick] (0,0) -- (1,0) (5,1) -- (5,2);
        \filldraw[standardcolor] (5,2) circle (1.5pt);
        \filldraw[standardcolor, fill=white] (5,1) circle (1.5pt);
        \draw[gray, dashed] (1,0) -- (5,1.33) (0,0) -- (5,1.67);
    \end{tikzpicture}\hspace{2cm}
    \begin{tikzpicture}
        \draw[standardcolor, very thick] (0,0) -- (1,0) (2,2) -- (2,3);
        \filldraw[standardcolor] (1,0) circle (1.5pt);
        \filldraw[standardcolor, fill=white] (0,0) circle (1.5pt);
        \draw[gray, dashed] (0.2,0) -- (2,3) (0.8,0) -- (2,2);
    \end{tikzpicture}\\[1cm]
    \begin{tikzpicture}
        \draw[decoratedcolor, very thick] (0,0) -- (1,0);
        \draw[standardcolor, very thick] (3,2) -- (4,2);
        \filldraw[decoratedcolor] (1,0) circle (1.5pt);
        \filldraw[decoratedcolor, fill=white] (0,0) circle (1.5pt);
        \draw[gray, dashed] (0.2,0) -- (1,0.8) -- (3,2) (0.8,0) -- (1,0.2) -- (4,2);
        \node[circle, fill=white, inner sep=0.1, text=decoratedcolor] at (0.5, 0.2) {$\ast$};
    \end{tikzpicture}\hspace{2cm}
    \begin{tikzpicture}
        \draw[decoratedcolor, very thick] (0,0) -- (1,0);
        \draw[standardcolor, very thick] (2,3) -- (3,3);
        \filldraw[standardcolor] (2,3) circle (1.5pt);
        \filldraw[standardcolor, fill=white] (3,3) circle (1.5pt);
        \draw[gray, dashed] (0,0) -- (1,1) -- (2.2,3) (1,0) -- (2.8,3);
        \node[circle, fill=white, inner sep=0.1, text=decoratedcolor] at (0.5, 0.2) {$\ast$};
    \end{tikzpicture}%
    \caption{Graphical description of the pairs of steps in $D_+$ (top left), $D_-$ (top right), $D_+^\ast$ (bottom left), and $D_-^\ast$ (bottom right). The dashed lines denote diagonal projections: in particular, they are vertical translates of the broken diagonal and might not be straight.}
    \label{fig:cdinv}
\end{figure}
\begin{figure}
    \centering
    \begin{tikzpicture}
        \draw[decoratedcolor, very thick] (1,0) -- (1,1);
        \draw[standardcolor, very thick] (3,2) -- (4,2);
        \filldraw[decoratedcolor, fill=white] (1,1) circle (1.5pt);
        \filldraw[decoratedcolor] (1,0) circle (1.5pt);
        \draw[gray, dashed] (1,0.8) -- (3,2) (1,0.2) -- (4,2);
        \node[circle, fill=white, inner sep=0.1, text=decoratedcolor] at (1.2, 0.5) {$\ast$};
    \end{tikzpicture}\hspace{2cm}
    \begin{tikzpicture}
        \draw[decoratedcolor, very thick] (1,0) -- (1,1);
        \draw[standardcolor, very thick] (2,3) -- (3,3);
        \filldraw[standardcolor] (2,3) circle (1.5pt);
        \filldraw[standardcolor, fill=white] (3,3) circle (1.5pt);
        \draw[gray, dashed] (1,1) -- (2.2,3) (1,0) -- (2.8,3);
        \node[circle, fill=white, inner sep=0.1, text=decoratedcolor] at (1.2, 0.5) {$\ast$};
    \end{tikzpicture}%
    \caption{Graphical description of the pairs of steps in $D_+^\ast$ (left), and $D_-^\ast$ (right) in the ENS representation.}
    \label{fig:cdinv_ENS}
\end{figure}

The set $D_+$ in \Cref{def:cdinv} consists of all pairs of vertical and horizontal steps $(i, j)$ such that the vertical step appears after the horizontal step and both endpoints of the horizontal step diagonally project onto the vertical step (top-most endpoint included, bottom-most endpoint excluded);
this can only happen for $n < m$ and if step $j$ is not decorated.

The set $D_-$ consists of all pairs of vertical and horizontal steps $(i, j)$ such that the vertical step is above the horizontal step and both endpoints of the vertical step diagonally project onto the horizontal step
(right-most endpoint included, left-most endpoint excluded);
this can only happen for $n > m$ and if step $j$ is not decorated.

The set $D_+^\ast$ consists of all pairs of horizontal steps $i < j$ such that step $i$ is decorated and
both endpoints of step $j$ diagonally project onto step $i$ (right-most endpoint included, left-most endpoint excluded);
this can only happen for $n < m$ and if step $j$ is not decorated.

The set $D_-^\ast$ consists of all pairs of horizontal steps $i < j$ such that step $i$ is decorated and
both endpoints of step $i$ diagonally project onto step $j$ (right-most endpoint excluded, left-most endpoint included);
this can only happen for $n > m$ and if step $j$ is not decorated.

See Figure~\ref{fig:cdinv}  (cf.\ Figure~\ref{fig:cdinv_ENS}) for an illustration of the definitions of $D_+$, $D_-$, $D_+^\ast$, and $D_-^\ast$.

It is convenient to give an alternative, more compact, definition of the dinv correction.

\begin{definition}
    \label{def:cdinv_alt}
    Let $(\pi, \decoratedsteps)$ be a fall-decorated rectangular path, and let
    \begin{align*}
        C_+ & \coloneqq \left\{ (i, j) \in \verticalsteps \times \horizontalsteps \mid j < i, \; r_j = 0, \text{ and } v_i < v_j \leq v_i + 1 - \textstyle\frac{n}{m} \right\} \\
        C_- & \coloneqq \left\{ (i, j) \in \verticalsteps \times \horizontalsteps \mid j < i \text{ and } v_j \leq v_i < v_j + \textstyle\frac{n}{m} + r_j - 1 \right\} \\
        C^\ast & \coloneqq \left\{ (i, j) \in \decoratedsteps \times (\horizontalsteps \sqcup \decoratedsteps) \mid i < j-1 \text{ and } v_i \leq v_j^+ < v_i + 1 \right\}
    \end{align*}
    where $v_j^+ \coloneqq v_j + \chi_{\decoratedsteps}(j) + \textstyle\frac{n}{m} \chi_{\horizontalsteps}(j)$ is the vertical distance between the left-most point of the horizontal step $j$ and the broken diagonal, which depends on whether or not $j$ is decorated. Here, $\chi_{A}$ denotes the indicator function of the set $A$.
    
    We define the \emph{dinv correction} of $(\pi, \decoratedsteps)$ as
    \[ \cdinv(\pi, \decoratedsteps) \coloneqq \# C_+ - \# C_- + \# C^\ast + \# B. \]
\end{definition}

\begin{figure}
    \centering
    \begin{tikzpicture}
        \draw[standardcolor, very thick] (0,0) -- (1,0) (5,1) -- (5,2);
        \draw[standardcolor, thick, dotted] (-1,0) -- (0,0) -- (0,-1);
        \draw[decoratedcolor] (-0.5, 0.25) node {{\small (no $\ast$)}};
        \filldraw[standardcolor] (5,2) circle (1.5pt);
        \filldraw[standardcolor, fill=white] (5,1) circle (1.5pt);
        \draw[gray, dashed] (1,0) -- (5,1.33) (0,0) -- (5,1.67);
    \end{tikzpicture}\\[1cm]
    \begin{tikzpicture}
        \draw[standardcolor, very thick] (3,0) -- (4,0) (5,2) -- (5,3);
        \draw[decoratedcolor, very thick] (0,0) -- (3,0);
        \filldraw[standardcolor] (4,0) circle (1.5pt);
        \filldraw[decoratedcolor, fill=white] (0,0) circle (1.5pt);
         \draw[decoratedcolor]
            (0.5,0.2) node {$\ast$}
            (1.5,0.2) node {$\ast$}
            (2.5,0.2) node {$\ast$};
        \draw[gray, dashed] (2.1,0) -- (5,3) (3.1,0) -- (5,2);
    \end{tikzpicture}\hspace{2cm}
    \begin{tikzpicture}
        \draw[decoratedcolor, very thick] (0,0) -- (1,0);
        \draw[standardcolor, very thick] (2.5,3) -- (3.5,3);
        \draw[decoratedcolor, very thick] (2.5, 3.08) -- (3.5, 3.08);
        \filldraw[decoratedcolor] (1,0) circle (1.5pt);
        \filldraw[decoratedcolor, fill=white] (0,0) circle (1.5pt);
        \draw[decoratedcolor] (3, 3.28) node {$\ast$};
        \draw[decoratedcolor] (0.5,0.2) node {$\ast$};
        \draw[gray, dashed] (0.6,0) -- (2.5,3);
    \end{tikzpicture}%
    \caption{Graphical description of the pairs of steps in $C_+$ (top), $C_-$ (bottom left), and $C^\ast$ (bottom right). The dashed lines denote the diagonal projections: in particular they are vertical translates of the broken diagonal and might not be straight.}
    \label{fig:cdinv_alt}
\end{figure}

The set $C_+$ in \Cref{def:cdinv_alt} contains the same pairs as the set $D_+$ in \Cref{def:cdinv}, with the additional restriction that the horizontal step is not immediately preceded by a decorated fall.

The set $C_-$ contains the same pairs as the set $D_-$ in \Cref{def:cdinv}, with the relaxed condition that the endpoints of the vertical step diagonally project onto the horizontal step or any of the decorated falls immediately preceding the horizontal step; the horizontal step is required to be non-decorated.

The set $C^\ast$ consists of all pairs of non-consecutive horizontal steps $i < j - 1$ such that step $i$ is decorated and the left-most endpoint of step $j$ diagonally projects onto step $i$ (right-most endpoint included, left-most endpoint excluded).
In the definition, it is important to specify $i < j-1$ (and not simply $i < j$) because otherwise $C^\ast$ would also contain all pairs of consecutive decorated falls, which we want to exclude.

See Figure~\ref{fig:cdinv_alt} (cf.\ Figure~\ref{fig:cdinv_alt_ENS}) for an illustration of the definitions of $C_+$, $C_-$, and $C^\ast$.

\begin{figure}
    \centering
    \begin{tikzpicture}
        \draw[standardcolor, very thick, sharp <-, sharp angle = 45] (3,0) -- (4,0);
        \draw[standardcolor, very thick] (5,2) -- (5,3);
        \draw[decoratedcolor, very thick,        -sharp >, sharp angle = 45] (3,3) -- (3,0);
        \filldraw[standardcolor] (4,0) circle (1.5pt);
        \filldraw[decoratedcolor, fill=white] (3,3) circle (1.5pt);
         \draw[decoratedcolor]
            (3.2,0.5) node {$\ast$}
            (3.2,1.5) node {$\ast$}
            (3.2,2.5) node {$\ast$};
        \draw[gray, dashed] (3,0.93) -- (5,3) (3.1,0) -- (5,2);
    \end{tikzpicture}\hspace{3cm}
    \begin{tikzpicture}
        \draw[decoratedcolor, very thick] (1,1) -- (1,0);
        \draw[standardcolor, very thick] (2.5,3) -- (3.5,3);
        \filldraw[decoratedcolor] (1,0) circle (1.5pt);
        \filldraw[decoratedcolor, fill=white] (1,1) circle (1.5pt);
        \draw[decoratedcolor] (1.2,0.5) node {$\ast$};
        \draw[gray, dashed] (1,0.6) -- (2.5,3);
    \end{tikzpicture}\hspace{2cm}
    \begin{tikzpicture}
        \draw[decoratedcolor, very thick] (1,1) -- (1,0);
        \draw[decoratedcolor, very thick] (2.5, 3) -- (2.5, 2);
        \filldraw[decoratedcolor] (1,0) circle (1.5pt);
        \filldraw[decoratedcolor, fill=white] (1,1) circle (1.5pt);
        \draw[decoratedcolor] (2.7, 2.5) node {$\ast$};
        \draw[decoratedcolor] (1.2,0.5) node {$\ast$};
        \draw[gray, dashed] (1,0.6) -- (2.5,3);
    \end{tikzpicture}%
    \caption{Graphical description of the pairs of steps in $C_-$ (left), and $C^\ast$ (center and right) in the ENS representation.}
    \label{fig:cdinv_alt_ENS}
\end{figure}

\begin{lemma}
    \label{lem:E_def}
    For a path $(\pi, \decoratedsteps) \in \RP(m+k,n+k)_{*k}$, let \[ E = \left\{ (i, j) \in \verticalsteps \times \decoratedsteps \mid j < i \text{ and } v_i < v_j \leq v_i + 1 \right\}, \]
    i.e., the set of pairs of vertical and horizontal steps $(i, j)$ such that the horizontal step is decorated, the vertical step appears after the horizontal step, and the right-most endpoint of the horizontal step diagonally projects onto the vertical step (top-most endpoint excluded, bottom-most endpoint included). Then
    \[ \# E = \sum_{\substack{(i,j) \in \verticalsteps \times \horizontalsteps \\ j < i}} \chi_{[\frac{n}{m}-1,\frac{n}{m}+r_j-1)}(v_i - v_j). \]
\end{lemma}

\begin{proof}
    Let $(i,j) \in \verticalsteps \times \horizontalsteps$, suppose that $r_j \neq 0$, and take $1 \leq p \leq r_j$. By construction, we have $j-p \in \decoratedsteps$, and $j < i$ holds if and only if $j-p < i$, in which case $v_{j-p} = v_j + \frac{n}{m} + p - 1$.
    Therefore
    \[
        \chi_{[-1,0)}(v_i-v_{j-p}) = \chi_{[\frac{n}{m}+p-2,\frac{n}{m}+p-1)}(v_i - v_j).
    \]
    Now,
    \begin{align*}
        \# E & = \sum_{\substack{(i,j) \in \verticalsteps \times \decoratedsteps \\ j < i}} \chi_{[-1,0)}(v_i-v_j) \\
             & = \sum_{\substack{(i,j) \in \verticalsteps \times \horizontalsteps \\ j < i}} \, \sum_{p=1}^{r_j} \chi_{[\frac{n}{m}+p-2,\frac{n}{m}+p-1)}(v_i - v_j) \\
             & = \sum_{\substack{(i,j) \in \verticalsteps \times \horizontalsteps \\ j < i}} \chi_{[\frac{n}{m}-1,\frac{n}{m}+r_j-1)}(v_i - v_j). \qedhere
    \end{align*}
\end{proof}

\begin{proposition}
\label{prop:def_of_cdinv_coincide}
    For each path $(\pi, \decoratedsteps) \in \RP(m+k,n+k)_{*k}$ we have
    \[ \# D_+ - \# D_- + \# D_+^\ast -\# D_-^\ast  -\# B^\ast = \# C_+ - \# C_- + \# C^\ast.\]
    In particular, \Cref{def:cdinv,def:cdinv_alt} are equivalent.
\end{proposition}

\begin{proof}
    Define $E$ as in \Cref{lem:E_def}.
    The statement is implied by the following two equalities:
    \[ \# D_- + \#E = \# D_+ - \# C_+ + \# C_- \quad \text{and} \quad \# C^\ast - \#D_+^\ast + \#D_-^\ast = \# E - \# B^\ast . \]
    Consider the first equality. By \Cref{def:cdinv,def:cdinv_alt} and \Cref{lem:E_def}, we have
    \begin{align}
        \# D_- & = \sum_{\substack{(i,j) \in \verticalsteps \times \horizontalsteps \\ j < i}} \chi_{[0,\frac{n}{m}-1)}(v_i-v_j) \label{eq:d_minus} \\
        \# E & = \sum_{\substack{(i,j) \in \verticalsteps \times \horizontalsteps \\ j < i}} \chi_{[\frac{n}{m}-1,\frac{n}{m}+r_j-1)}(v_i - v_j) \label{eq:e} \\
        \# D_+ - \# C_+ &= \sum_{\substack{(i,j) \in \verticalsteps \times \horizontalsteps \\ j < i}} (1 - \delta_{r_j,0}) \, \chi_{[\frac{n}{m}-1,0)}(v_i-v_j) \label{eq:d_plus_minus_c_plus} \\
        \# C_- & = \sum_{\substack{(i,j) \in \verticalsteps \times \horizontalsteps \\ j < i}} \chi_{[0,\frac{n}{m}+r_j-1)}(v_i-v_j). \label{eq:c_minus}
    \end{align}
    Now, fix a pair $(i,j) \in \verticalsteps \times \horizontalsteps$ with $j < i$. If $r_j = 0$, then the corresponding terms in \eqref{eq:e} and \eqref{eq:d_plus_minus_c_plus} vanish, and the terms in \eqref{eq:d_minus} and \eqref{eq:c_minus} are equal. If $r_j \neq 0$ and $n < m$, the term in \eqref{eq:d_minus} vanishes, and the term in \eqref{eq:e} is the sum of the terms in \eqref{eq:d_plus_minus_c_plus} and \eqref{eq:c_minus}. If $r_j \neq 0$ and $n \geq m$, the term in \eqref{eq:d_plus_minus_c_plus} vanishes, and the term in \eqref{eq:c_minus} is the sum of the terms in \eqref{eq:d_minus} and \eqref{eq:e}.
    This proves the first equality.

    Consider now the second equality. Let
    \[ F = \left\{ (i, j) \in \decoratedsteps \times (\horizontalsteps \sqcup \decoratedsteps) \mid i < j-1 \text{ and }  v_j < v_i \leq v_j^+ \right\}, \]
    and note that
    \[ F \setminus D^\ast_- = \left\{ (i, j) \in \decoratedsteps \times (\horizontalsteps \sqcup \decoratedsteps) \mid i < j-1 \text{ and }  v_j < v_i \leq v_j^+ < v_i + 1 \right\} = C^\ast \setminus D_+^\ast. \]
    Since $D_-^\ast \subseteq F$ and $D_+^\ast \subseteq C^\ast$, we have $\# F = \#C^\ast - \#D_+^\ast + \#D_-^\ast$ and so it is enough to show that $\# E = \# F + \# B^\ast$. Let $(i,j) \in E$; this means that the vertical translate of the broken diagonal passing through the right-most endpoint of $j$ crosses the vertical step $i$ (top-most endpoint excluded, bottom-most endpoint included) at height $v_j$.
    
    If the same vertical translate of the broken diagonal crosses another horizontal step, let $j'> i > j$ be the smallest such step; then $(j,j') \in F$ and this correspondence is one-to-one. If it does not cross any other horizontal step, then we necessarily have $v_j \leq 0$, and for each $j$ such that $v_j \leq 0$ there exists exactly one $i$ with this property, namely $\max \{ i \in \verticalsteps \mid (i,j) \in E \}$. The second equality follows.
\end{proof}

\begin{definition}
    We define the \emph{dinv} of a labeled rectangular path $(\pi, w)$ as \[ \dinv(\pi, w) \coloneqq \tdinv(\pi, w) + \cdinv(\pi).
    \]
\end{definition}
Note that cdinv only depends on the path $\pi$, whereas tdinv also depends on the labels~$w$.

\begin{remark}
    When $k=0$ (i.e.\ there are no decorated falls), the sets $C_+$ and $C_-$ of \Cref{def:cdinv_alt} reduce to the sets in \cite{HicksSergel2015SimplerFormulaNumber}*{Theorem 2 and Figure 3}, which cannot be simultaneously non-empty, and $C_*$ is empty.
    Also note that, when $m=n$ (i.e.\ the path is square), the sets $D_+, D_-, D_+^\ast$, and $D_-^\ast$ of \Cref{def:cdinv} are empty, so the cdinv reduces to $\# B - \# B^\ast$.
    In particular, our dinv statistic extends both the dinv of rectangular (Dyck) paths as defined in \cite{HicksSergel2015SimplerFormulaNumber} and \cite{IraciPagariaPaoliniVandenWyngaerd2023}, and the dinv for rise-decorated Dyck paths of the rise Delta theorem in \cites{HaglundRemmelWilson2018DeltaConjecture, DAdderioMellit2022CompositionalDelta}.
\end{remark}

\section{Combinatorial arguments}
\label{sec:combinatorial-arguments}

\subsection{Combinatorial interpretation of the skewing operator}
\label{ssc:skewing}

Recall that
\begin{equation}
    \label{eq:skewing_x_plus_y}
    g[Y]^\perp f[X+Y] \big|_{Y=0} = g[X]^\perp f[X]
\end{equation}
for any symmetric functions $f, g \in \Lambda$. Readers unfamiliar with plethystic notation can refer to \cite{LoehrRemmel2011ExposePlethystic} for details (cf.\ \cite{Macdonald1995Book}); using the results in \cite{LoehrRemmel2011ExposePlethystic}*{Section~3.2}, \eqref{eq:skewing_x_plus_y} is immediate when $f, g$ are Schur functions, and holds by linearity in the general case.

Let $\alpha$ be any weak composition. We want to apply the skewing operator $h_\alpha^\perp$ to the expression \[ \sum_{\pi \in \LRP(m,n)} q^{\dinv(\pi)} t^{\area(\pi)} x^\pi. \] This operation is well-defined as the expression is a positive sum of LLT polynomials, and so it is a (Schur positive) symmetric function. Using \eqref{eq:skewing_x_plus_y}, we know that \[ h_\alpha^\perp \sum_{\pi \in \LRP(m,n)} q^{\dinv(\pi)} t^{\area(\pi)} x^\pi = h_\alpha[Y]^\perp \left. \left( \sum_{\pi \in \LRP(m,n)} q^{\dinv(\pi)} t^{\area(\pi)} x^\pi \right)\left[X+Y\right] \right\rvert_{Y=0}. \]
The plethystic substitution in the combinatorial term corresponds to the replacement of the alphabet $\mathbb{Z}_+$ with $\mathbb{Z}_+ \cup \overline{\mathbb{Z}}_+$ in the labeling, where we can assume that the labels in $\mathbb{Z}_+$ are smaller than the labels in $\overline{\mathbb{Z}}_+$, since the expression is symmetric.

Now, since homogeneous and monomial symmetric functions are dual to each other, the skewing operator $h_\alpha[Y]^\perp$ followed by the substitution $Y=0$ isolates the terms of the form $f[X] m_\alpha[Y]$ for some $f \in \Lambda$; since the formula is symmetric in $Y$, the coefficient of $m_\alpha[Y]$ is the same as that of $y^\alpha$ when the expression is interpreted as a formal power series in $Y$ with coefficients in $\Lambda$.

For analogous reasons, the same holds for $\LRD(m,n)$. This motivates the following definition.

\begin{definition}
    For $m, n, r \in \mathbb{N}$, and a weak composition $\alpha \vDash r$, we define the sets $\LRD(m,n)^\alpha$ and $\LRP(m,n)^{\alpha}$ to be the sets of rectangular (Dyck) paths of size $m \times n$ with labels in the alphabet $\mathbb{Z}_+ \cup \overline{\mathbb{Z}}_+$ such that there are exactly $\alpha_i$ vertical steps with label $\overline{i}$ for all $i$.
    We refer to the labels in $\mathbb{Z}_+$ as \emph{small} and to the labels in $\overline{\mathbb{Z}}_+$ as \emph{big}.
    For $\pi \in \LRP(m,n)^{\alpha}$, we denote by $\verticalsteps$ the set of vertical steps with small labels and by $\bigsteps$ the set of vertical steps with big labels, so that the set of all steps of $\pi$ is the disjoint union $\horizontalsteps \sqcup \verticalsteps \sqcup \bigsteps$ (totally ordered by traversing the path).    
\end{definition}

As before, set $x^\pi = x^w = \prod_{i \in \verticalsteps} x_{w_i}$; note that this product disregards the big labels.
From the previous discussion, we deduce the following.

\begin{proposition}
    \label{prop:skewing-h}
    For any $m, n, r \in \mathbb{N}$ and $\alpha \vDash r$, we have
    \[ h_\alpha^\perp \sum_{\pi \in \LRP(m,n)} q^{\dinv(\pi)} t^{\area(\pi)} x^\pi = \sum_{\pi \in \LRP(m,n)^{\alpha}} q^{\dinv(\pi)} t^{\area(\pi)} x^\pi, \]
    and the same holds if we replace $\LRP(m,n)$ with $\LRD(m,n)$.
\end{proposition}

We now want to apply the skewing operator $s_{(m-1)^k}^\perp$ to $e_{m,n+km}$. In order to do this, we need the homogeneous expansion of $s_{(m-1)^k}$. Let us recall the definition of \emph{allowable composition}.

\begin{definition}[{\cite{GillespieGorskyGriffin2025}*{Definition~4.22}}]
    Let $\alpha \vDash k$ be a weak composition with $\ell(\alpha) \leq k$. We say that $\alpha$ is \emph{allowable} if $\alpha_1 > 0$, and each $\alpha_i > 0$ is followed by exactly $\alpha_i - 1$ zeros within the first $k$ entries (i.e.\ $\alpha_i + i - 1 \leq k$ for all $i$, and $\alpha_{i+1} = \dots = \alpha_{i+\alpha_i-1} = 0$ and $\alpha_{i+\alpha_i} \neq 0$ unless $i+\alpha_i - 1 = k$).

    We also define the \emph{sign} of an allowable composition as
    \[ \sgn(\alpha) \coloneqq \sum_{\alpha_i > 0} (\alpha_i - 1) = \# \{1 \leq i \leq k \mid \alpha_i = 0\}. \]
\end{definition}

\begin{remark}
    Allowable compositions of $k$ describe valid rankings of $k$ competitors in a tournament, with ties allowed: if there are $\alpha_i$ competitors tied for $i\th$ place, then the following $\alpha_i - 1$ rankings are not available. Words with \emph{allowable content} (i.e., words such that the content is an allowable composition) are also known as \emph{Fubini rankings}.
    We refer to \href{https://oeis.org/A000670}{OEIS A000670} for related combinatorics.
    
    For example, the word $w = 6313136$ is a Fubini ranking corresponding to a tournament in which the first competitor ranks $6\th$, the second competitor ranks $3\rd$, and so on. There is a two-way tie between competitors $3$ and $5$ for the $1\st$ place, a three-way tie between competitors $2, 4$, and $6$ for the $3\rd$ place, and again a two-way tie between competitors $1$ and $7$ for the $6\th$ place.
\end{remark}

Using the Jacobi-Trudi formula, we get the following.

\begin{proposition}[{\cite{GillespieGorskyGriffin2025}*{Equation~(10)}}]
    \label{prop:schur-expansion}
    For $m, k \in \mathbb{N}$, we have \[ s_{(m-1)^k} = \sum_{\alpha \text{ allowable}} (-1)^{\sgn(\alpha)} h_{\tilde\alpha} \pmod{h_j \mid j > m}, \]
    where $\tilde{\alpha}_i = m - \alpha_i$, and $\sgn(\alpha) = \# \{1 \leq i \leq k \mid \alpha_i = 0\}$ (recall that $h_{a} = 0$ for all $a < 0$).
\end{proposition}

\subsection{Bijection with rectangular Dyck paths}
\label{ssc:bijection}

Let $m, n, k \in \mathbb{N}$ and $\alpha \vDash k$ allowable. Let $\tilde{\alpha} \vDash k(m-1)$ be defined as $\tilde{\alpha}_i = m - \alpha_i$. In order to compute $s_{(m-1)^k}^\perp e_{m,n+km}$ in terms of our objects, we need a bijection between the set of paths in $\LRP(m,n+km)^{\tilde\alpha}$ and a subset of paths in $\LRP(m+k,n+k)_{\ast k}$ whose labels depend on $\alpha$ (see~\Cref{prop:shape-bijection}). This map is an adaptation of \cite{GillespieGorskyGriffin2025}*{Definition~4.6}. Note that both sets are empty if $\alpha_i > m$ for some $i$.

From now on, to improve readability, we will use $\tilde\pi$ to denote paths of size $m \times (n+km)$ and reserve $\pi$ for (decorated) paths of size $(m+k) \times (n+k)$.

\begin{definition}
    \label{def:shape-bijection}
    Let $\tilde\pi \in \LRP(m,n+km)^{\tilde\alpha}$.
    We define $\psi_0(\tilde\pi) \in \LRP(m+k,n+k)_{\ast k}$ as the path obtained from $\tilde\pi$ by the following procedure:
    \begin{enumerate}
        \item delete all vertical steps $i \in \bigsteps$; %
        \item if immediately before the horizontal step $j \in \horizontalsteps$ there were $a_j$ such steps, replace them with $k-a_j$ decorated falls;
        \item keep everything else in place, including the remaining labels.
    \end{enumerate}      
\end{definition}

\begin{remark}
\label{rmk:fall_labels}
    Via this map, we lose the information given by the big labels.
However, since the labels in each column are a subset of $[\overline{k}]$, and they get replaced by a streak of decorated falls of complementary size, we can recover the information by labeling the decorated falls of $\psi_0(\tilde\pi)$ with the complementary subset of $[\overline{k}]$.
\end{remark}
 This leads to the following definition.

\begin{definition}
    Let $\pi \in \LRP(m+k,n+k)_{\ast k}$. We define a fall-labeling of $\pi$ to be a function $\overline{w} \colon \decoratedsteps \to [\overline{k}] \coloneqq \{\overline{1}, \dots, \overline{k}\}$ such that, if $i, i+1 \in \decoratedsteps$, then $\overline{w}_i < \overline{w}_{i+1}$.
\end{definition}

In other words, we assign a label from $\overline{1}$ to $\overline{k}$ to each decorated fall of the path, in such a way that the labels are strictly increasing from left to right.
The \emph{content} of a fall-labeling is the weak composition $\alpha \vDash k$ where
\[ \alpha_i \coloneqq \# \{ j \in \decoratedsteps \mid \overline{w}_j = \overline{i} \}, \]
that is, the number of times the label $i$ appears on a fall.

\begin{definition}
    Given a fall-labeling $\overline{w}$ of $\pi$, we define the \emph{star word} of $\pi$ as the sequence $w^\ast(\pi, \overline{w})$ of labels of the decorated falls, read from the greatest to the smallest vertical distance, and from right to left in case of a tie. We say that a fall-labeling is \emph{allowable} if it is a Fubini ranking, that is, $w^\ast(\pi, \overline{w})$ has allowable content. See \Cref{fig:star-word} for an example.
\end{definition}

\begin{figure}
    \centering
    \begin{tikzpicture}[scale=0.6]
        \begin{scope}
        
            \draw[step=1.0, gridcolor, thin] (0, 0) grid (12, 8);
    
            \begin{scope}
                \clip (0,0) rectangle (12, 8);
                \draw[diagonalcolor, thin]
                    (0, 0) -- (1, 0.5) -- (2, 1.5) -- (4, 2.5) -- (6, 4.5) -- (8, 5.5) -- (9, 6.5) -- (12, 8);
            \end{scope}
    
            \begin{scope}[standardcolor, very thick]
                \draw[-sharp >, sharp angle = -45] (0, 0) -- (0, 1) -- (1, 1) -- (1, 3);
                \draw (2, 3) -- (3, 3) -- (3, 4) -- (4, 4);
                \draw (6, 4) -- (7, 4) -- (7, 6) -- (8, 6);
                \draw (9, 6) -- (11, 6) -- (11, 8) -- (12, 8);
            \end{scope}
    
            \begin{scope}[decoratedcolor, very thick]
                \draw[sharp <-, sharp angle = -45] (1, 3) -- (2, 3);
                \draw (4, 4) -- (6, 4);
                \draw (8, 6) -- (9, 6);
            \end{scope}

            \begin{scope}[decoratedcolor]
                \draw (1.7, 3.3) node {$\ast$};
                \draw (4.7, 4.3) node {$\ast$};
                \draw (5.7, 4.3) node {$\ast$};
                \draw (8.7, 6.3) node {$\ast$};
            \end{scope}
    
            \begin{scope}[every node/.style = {font=\small}]
                \draw[decoratedcolor]
                    (1.3, 3.35) node {$\overline{2}$}
                    (4.25, 4.35) node {$\overline{2}$}
                    (5.25, 4.35) node {$\overline{4}$}
                    (8.25, 6.35) node {$\overline{1}$}
                    ;

                \draw[standardcolor]
                    (-0.25, 0.5) node {$2$}
                    (0.75, 1.5) node {$1$}
                    (0.75, 2.5) node {$4$}
                    (2.75, 3.5) node {$5$}
                    (6.75, 4.5) node {$2$}
                    (6.75, 5.5) node {$4$}
                    (10.75, 6.5) node {$1$}
                    (10.75, 7.5) node {$5$}
                    ;
            \end{scope}
        \end{scope}
    \end{tikzpicture}
    \caption{A path $\pi \in \LRP(12,8)_{\ast 4}$ with a fall-labeling $\overline{w}$.
    Left to right, the vertical distances of the decorated falls from the broken diagonal are: $\frac32, \frac12, -\frac12, -\frac12$.
    Therefore, the star word is $w^\ast(\pi, \overline{w}) = \overline{2}\,\overline{2}\,\overline{1}\,\overline{4}$; here, the label $\overline{1}$ appears before $\overline{4}$ in the star word because the step labeled $\overline{1}$ appears after the step labeled $\overline{4}$.
    The fall-labeling is allowable because the content of $w^\ast(\pi, \overline{w})$ is $(1,2,0,1)$, which is allowable.}
    \label{fig:star-word}
\end{figure}

Let $\FL(\pi)$ denote the set of allowable fall-labelings of $\pi$, and $\FL(\pi, \alpha)$ the subset of fall-labelings with content $\alpha$. The following proposition is immediate.

\begin{proposition}
    \label{prop:shape-bijection}
    The map $\psi_0$ lifts to a bijection $\psi$ between $\LRP(m,n+km)^{\tilde\alpha}$ and the set \[ \{ (\pi, \overline{w}) \mid \pi \in \LRP(m+k,n+k)_{\ast k}, \; \overline{w} \in \FL(\pi, \alpha) \}, \]
    where $\overline{w}$ is obtained as in \Cref{rmk:fall_labels}.
\end{proposition}

The definitions of $\dinv$, $\tdinv$, and all the other statistics extend to objects with a fall-labeling by disregarding it.

To better visualize the bijection $\psi$, we introduce an ENS representation for paths in $\LRP(m, n+km)^{\tilde\alpha}$ as follows: before every horizontal step $j$, insert $k$ South steps, where the first $a_j$ are non-decorated and the last $k - a_j$ are decorated (here, $a_j$ is as in \Cref{def:shape-bijection}).
The result is a lattice path from $(0, 0)$ to $(m, n)$ with $m$ East steps (indexed by $\horizontalsteps$), $n+km$ North steps ($n+k$ of them being indexed by $\verticalsteps$, and $k(m-1)$ being indexed by $\bigsteps$), and $km$ South steps.
Denote by $\decoratedsteps$ the set of decorated South steps and by $\southsteps$ the set of non-decorated South steps (naturally in bijection with $\bigsteps$); we have $\# \decoratedsteps = k$ and $\#\southsteps = \#\bigsteps = k(m-1)$.
See \Cref{fig:psi} (top right).

When passing to the ENS representation, the vertical distance between the left-most endpoint of horizontal steps and the main diagonal is reduced by $k$; the vertical distances are preserved for all other endpoints.
As illustrated in \Cref{fig:psi}, the map $\psi$ operates on the ENS representation simply by pruning away the vertical steps in $\bigsteps \sqcup \southsteps$.
The remaining steps $i \in \horizontalsteps \sqcup \verticalsteps \sqcup \decoratedsteps$ are identified with the corresponding steps in $\psi(\tilde\pi)$; when working with both $\tilde\pi$ and $(\pi, \overline{w}) = \psi(\tilde\pi)$, we will use the sets $\horizontalsteps, \verticalsteps, \decoratedsteps$ to denote steps from both paths.

\begin{figure} %
    \centering
    \begin{tikzpicture}[scale=.6]
        \begin{scope}[shift={(0, 0)}]
            \fill[areacolor] (1, 7) rectangle (2, 8);
            \fill[areacolor] (2, 11) rectangle (3, 12);
            
            \draw[step=1.0, gridcolor, thin] (0,0) grid (5,17);
    
            \draw[diagonalcolor] (0, 0) -- (5, 17);
    
            \begin{scope}[standardcolor, very thick]
                \draw (0, 0) -- (0, 1);
                \draw[sharp <-, sharp angle = -45] (0, 4) -- (1, 4) -- (1, 6);
                \draw[sharp <-, sharp angle = -45] (1, 8) -- (2, 8) -- (2, 9);
                \draw[sharp <-, sharp angle = -45] (2, 12) -- (2.5, 12);
                \draw[-sharp >, sharp angle = 45] (2.5, 12) -- (3, 12);
                \draw[sharp <-, sharp angle = -45] (3, 13) -- (4, 13) -- (4, 14);
                \draw[sharp <-, sharp angle = -45] (4, 17) -- (5, 17);
            \end{scope}
    
            \begin{scope}[bigcolor, very thick, -sharp >, sharp angle = -45]
                \draw (0, 1) -- (0, 4);
                \draw (1, 6) -- (1, 8);
                \draw (2, 9) -- (2, 12);
                \draw[sharp <-, sharp angle = 45] (3, 12) -- (3, 12.5);
                \draw (3, 12.5) -- (3, 13);
                \draw (4, 14) -- (4, 17);
            \end{scope}
            
            \begin{scope}[every node/.style = {font=\scriptsize}]
                \draw[bigcolor]
                    (-0.25,1.5) node {$\overline{1}$}
                    (-0.25,2.5) node {$\overline{2}$}
                    (-0.25,3.5) node {$\overline{3}$}
                    (0.75,6.5) node {$\overline{2}$}
                    (0.75,7.5) node {$\overline{3}$}
                    (1.75,9.5) node {$\overline{1}$}
                    (1.75,10.5) node {$\overline{2}$}
                    (1.75,11.5) node {$\overline{3}$}
                    (2.7,12.5) node {$\overline{1}$}
                    (3.75,14.5) node {$\overline{1}$}
                    (3.75,15.5) node {$\overline{2}$}
                    (3.75,16.5) node {$\overline{3}$}
                    ;

                \draw[standardcolor]
                    (-0.25, 0.5) node {$2$}
                    (0.75, 4.5) node {$1$}
                    (0.75, 5.5) node {$4$}
                    (1.75, 8.5) node {$5$}
                    (3.75, 13.5) node {$2$}
                    ;
            \end{scope}
        \end{scope}

        \draw[->] (5.5, 14) to [out=20,in=160] node[midway, above] {\small ENS} (7.5, 14);

        \begin{scope}[shift={(8.5, 11)}]
            \fill[areacolor] (1, 1) rectangle (2, 2);
            \fill[areacolor] (2, 2) rectangle (3, 3);
        
            \draw[step=1.0, gridcolor, thin] (0,0) grid (5,6);
    
            \begin{scope}
                \clip (0,0) rectangle (5,2);
                \draw[diagonalcolor, thin]
                (0, 0) -- (5, 2);
            \end{scope}
    
            \begin{scope}[standardcolor, very thick]
                \draw (-0.05, 0) -- (-0.05,1);
                \draw[sharp <-, sharp angle = 45] (0.05, 1) -- (0.95, 1) -- (0.95, 3);
                \draw[sharp <-, sharp angle = 45] (1.05, 2) -- (1.95, 2) -- (1.95, 3);
                \draw[sharp <-sharp >, sharp angle = 45] (2.05, 3) -- (2.95, 3);
                \draw[sharp <-, sharp angle = 45] (3.05, 1) -- (3.95, 1) -- (3.95, 2);
                \draw[sharp <-, sharp angle = 45] (4.05, 2) -- (5,2);
            \end{scope}
    
            \begin{scope}[bigcolor, very thick]
                \draw (-0.05,1) -- (-0.05, 4) -- (0, 4);
                \draw (0.95, 3) -- (0.95, 5) -- (1, 5);
                \draw (1.95, 3) -- (1.95, 6) -- (2, 6);
                \draw[sharp <-, sharp angle = 45] (2.95, 3) -- (2.95, 4) -- (3, 4);
                \draw (3.95, 2) -- (3.95, 5) -- (4, 5);
            \end{scope}
            
            \begin{scope}[southcolor, very thick]
                \draw[-sharp >, sharp angle = 45] (0, 4) -- (0.05, 4) -- (0.05, 1);
                \draw (1, 5) -- (1.05, 5) -- (1.05, 3);
                \draw[-sharp >, sharp angle = 45] (2, 6) -- (2.05, 6) -- (2.05, 3);
                \draw (3, 4) -- (3.05, 4) -- (3.05, 3);
                \draw[-sharp >, sharp angle = 45] (4, 5) -- (4.05, 5) -- (4.05, 2);
            \end{scope}
    
            \begin{scope}[decoratedcolor, very thick]
                \draw[-sharp >, sharp angle = 45] (1.05, 3) -- (1.05, 2);
                \draw[-sharp >, sharp angle = 45] (3.05, 3) -- (3.05, 1);
            \end{scope}

            \begin{scope}[decoratedcolor]
                \draw (1.25,2.25) node {$\ast$};
                \draw (3.25,2.25) node {$\ast$};
                \draw (3.25,1.25) node {$\ast$};
            \end{scope}
    
            \begin{scope}[every node/.style = {font=\scriptsize}]
                \draw[bigcolor]
                    (-0.25,1.5) node {$\overline{1}$}
                    (-0.25,2.5) node {$\overline{2}$}
                    (-0.25,3.5) node {$\overline{3}$}
                    (0.75,3.5) node {$\overline{2}$}
                    (0.75,4.5) node {$\overline{3}$}
                    (1.75,3.5) node {$\overline{1}$}
                    (1.75,4.5) node {$\overline{2}$}
                    (1.75,5.5) node {$\overline{3}$}
                    (2.7,3.5) node {$\overline{1}$}
                    (3.75,2.5) node {$\overline{1}$}
                    (3.75,3.5) node {$\overline{2}$}
                    (3.75,4.5) node {$\overline{3}$}
                    ;
    
                \draw[decoratedcolor]
                    (1.25, 2.7) node {$\overline{1}$}
                    (3.25, 1.7) node {$\overline{3}$}
                    (3.25, 2.7) node {$\overline{2}$}
                    ;

                \draw[standardcolor]
                    (-0.25, 0.5) node {$2$}
                    (0.75, 1.5) node {$1$}
                    (0.75, 2.5) node {$4$}
                    (1.75, 2.5) node {$5$}
                    (3.75, 1.5) node {$2$}
                    ;

            \end{scope}
        \end{scope}

        \draw[->] (11, 10.5) -- node[midway, right] {\small $\psi$} (11, 9.5);

        \begin{scope}[shift={(8.5, 6)}]
            \fill[areacolor] (1, 1) rectangle (2, 2);
            \fill[areacolor] (2, 2) rectangle (3, 3);
        
            \draw[step=1.0, gridcolor, thin] (0,0) grid (5,3);
    
            \begin{scope}
                \clip (0,0) rectangle (5,2);
                \draw[diagonalcolor, thin]
                (0, 0) -- (5, 2);
            \end{scope}
    
            \begin{scope}[standardcolor, very thick]
                \draw (0, 0) -- (0, 1) -- (0.95, 1) -- (0.95, 3) -- (1, 3);
                \draw[sharp <-, sharp angle = 45] (1.05, 2) -- (2, 2) -- (2, 2.5);
                \draw[-sharp >, sharp angle = -45] (2, 2.5) -- (2, 3) -- (3, 3);
                \draw[sharp <-, sharp angle = 45] (3, 1) -- (4, 1) -- (4, 2) -- (5, 2);
            \end{scope}
    
            \begin{scope}[decoratedcolor, very thick]
                \draw[-sharp >, sharp angle = 45] (1, 3) -- (1.05, 3) -- (1.05, 2);
                \draw[sharp <-, sharp angle = -45] (3, 3) -- (3, 2);
                \draw[-sharp >, sharp angle = 45] (3, 2) -- (3, 1);
            \end{scope}

            \begin{scope}[decoratedcolor]
                \draw (1.25,2.25) node {$\ast$};
                \draw (3.25,2.25) node {$\ast$};
                \draw (3.25,1.25) node {$\ast$};
            \end{scope}
    
            \begin{scope}[every node/.style = {font=\scriptsize}]
                \draw[decoratedcolor]
                    (1.25, 2.7) node {$\overline{1}$}
                    (3.25, 1.7) node {$\overline{3}$}
                    (3.25, 2.7) node {$\overline{2}$}
                    ;

                \draw[standardcolor]
                    (-0.25, 0.5) node {$2$}
                    (0.75, 1.5) node {$1$}
                    (0.75, 2.5) node {$4$}
                    (1.75, 2.5) node {$5$}
                    (3.75, 1.5) node {$2$}
                    ;
            \end{scope}
        \end{scope}

        \draw[->] (14.5, 7) to [out=20,in=70] node[midway, above right] {\small ENS$^{-1}$} (15.5, 5.5);

        \begin{scope}[shift={(7, 0)}]
            \fill[areacolor] (2, 2) rectangle (3, 3);
            \fill[areacolor] (3, 3) rectangle (4, 4);
        
            \draw[step=1.0, gridcolor, thin] (0, 0) grid (8, 5);
    
            \begin{scope}
                \clip (0,0) rectangle (8, 5);
                \draw[diagonalcolor, thin]
                    (0, 0) -- (1, 0.4) -- (2, 1.4) -- (4, 2.2) -- (6, 4.2) -- (8, 5);
            \end{scope}
    
            \begin{scope}[standardcolor, very thick]
                \draw[-sharp >, sharp angle = -45] (0, 0) -- (0, 1) -- (1, 1) -- (1, 3);
                \draw (2, 3) -- (3, 3) -- (3, 4) -- (4, 4);
                \draw (6, 4) -- (7, 4) -- (7, 5) -- (8, 5);
            \end{scope}
    
            \begin{scope}[decoratedcolor, very thick]
                \draw[sharp <-, sharp angle = -45] (1, 3) -- (2, 3);
                \draw (4, 4) -- (6, 4);
            \end{scope}

            \begin{scope}[decoratedcolor]
                \draw (1.7, 3.3) node {$\ast$};
                \draw (4.7, 4.3) node {$\ast$};
                \draw (5.7, 4.3) node {$\ast$};
            \end{scope}
    
            \begin{scope}[every node/.style = {font=\scriptsize}]
                \draw[decoratedcolor]
                    (1.3, 3.35) node {$\overline{1}$}
                    (4.25, 4.35) node {$\overline{2}$}
                    (5.25, 4.35) node {$\overline{3}$}
                    ;

                \draw[standardcolor]
                    (-0.25, 0.5) node {$2$}
                    (0.75, 1.5) node {$1$}
                    (0.75, 2.5) node {$4$}
                    (2.75, 3.5) node {$5$}
                    (6.75, 4.5) node {$2$}
                    ;
            \end{scope}
        \end{scope}
    \end{tikzpicture}

    \caption{
        An example of the map $\psi$ in \Cref{def:shape-bijection} for $m=5$, $n=2$, $k=3$, and $\tilde\alpha = (4,4, 4)$.
        On the left, a path $\tilde\pi \in \LRP(m, n+km)^{\tilde\alpha}$ with its big labels (and the corresponding vertical steps $i \in \bigsteps$) highlighted in red.
        On the top right, the ENS representation of $\tilde\pi$.
        On the center right, the ENS representation of $(\pi, \overline{w}) = \psi(\tilde\pi)$; it is obtained from the ENS representation of $\tilde\pi$ by pruning away the steps in $\bigsteps \sqcup \southsteps$. 
        On the bottom right, the standard representation of $\pi$.
        In all paths, the two squares contributing to the area are highlighted in gray, and the fall-labeling $\overline{w}$ is shown next to the decorated steps.
        In this example, the star word is $w^\ast(\pi, \overline{w}) = \overline{1} \, \overline{2}\, \overline{3}$.
    }
    \label{fig:psi}
\end{figure}

From the previous description of $\psi$, the following lemma is immediate.

\begin{lemma}
    \label{lem:preserve-vdistance}
    Let $\tilde\pi \in \LRP(m,n+km)^{\tilde\alpha}$.
    For any step $i \in \horizontalsteps \sqcup \verticalsteps$ (horizontal, or vertical with a small label), its vertical distance is preserved by $\psi$:
    \[ v_i(\tilde\pi) = v_i(\psi(\tilde\pi)). \]
    In particular, we have $\area(\tilde\pi) = \area(\psi(\tilde\pi))$ and $\shift(\tilde\pi) = \shift(\psi(\tilde\pi))$.
    Therefore, $\psi$ restricts to a bijection between $\LRD(m,n+km)^{\tilde\alpha}$ and the set
    \[ \{ (\pi, \overline{w}) \mid \pi \in \LRD(m+k,n+k)_{\ast k}, \, \overline{w} \in \FL(\pi, \alpha) \}. \]
\end{lemma}

\begin{remark}
    \label{rmk:falldinv}
    Given $i, j \in \verticalsteps$, we have $i \rightarrow j$ in $\tilde\pi$ if and only if $i \rightarrow j$ in $(\pi, \overline{w}) = \psi(\tilde\pi)$.
    Therefore, the pair $(i, j)$ contributes to $\tdinv(\tilde\pi)$ if and only if it contributes to $\tdinv(\pi)$.
    The remaining contributions to $\tdinv(\tilde\pi)$ are due to attacking pairs $(i, j) \in (\verticalsteps \sqcup \bigsteps) \times \bigsteps$.
    We call $\falldinv(\pi, \overline{w})$ the number of such contributions, so that
    \[ \tdinv(\tilde\pi) = \tdinv(\pi) + \falldinv(\pi, \overline{w}). \]
\end{remark}

\subsection{Sign-reversing involution}
\label{ssc:involution}

Let $\pi \in \LRP(m+k,n+k)_{\ast k}$, and for $\overline{w} \in \FL(\pi)$, let $\sgn(\overline{w})$ be the sign of its content. We aim to define a sign-reversing involution on $\FL(\pi)$ such that $\overline{w}$ is a fixed point if and only if $w^\ast(\pi, \overline{w}) = \overline{1} \, \overline{2} \cdots \overline{k}$.

Let us recall the sign-reversing involution $\varphi$ on words with allowable content
from \cite{GillespieGorskyGriffin2025}.

\begin{definition}[{\cite{GillespieGorskyGriffin2025}*{Definition~4.28}}]
    Let $a$
    be any word of length $k$ with allowable content. Let $p$ be the largest entry in $a$ such that
    \begin{samepage}
    \begin{itemize}
        \item there is only one $p$, and
        \item if $q$ is the largest entry in $a$ that is less than $p$, then $p$ is to the left of every $q$.
    \end{itemize}
    \end{samepage}
    Notice that such a $p$ might not exist. Let $r$ be the largest repeated entry in $a$, and let $s$ be the smallest entry larger than $r$ in $a$, or $s = k+1$ if $r$ is already the largest. Note that $p \neq r$. Then we define $\varphi(a)$ as the word obtained from $a$ by applying the following procedure:
    \begin{itemize}
        \item[\textit{Case 1.}] If $p > r$ or $r$ does not exist, replace $p$ with an occurrence of $q$.
        \item[\textit{Case 2.}] If $p < r$ or $p$ does not exist, replace the first occurrence of $r$ with $s-1$.
        \item[\textit{Case 3.}] If neither $p$ nor $r$ exist, do nothing.
    \end{itemize}
\end{definition}

We have the following result.

\begin{theorem}[{\cite{GillespieGorskyGriffin2025}*{Theorem~4.30}}]
    \label{thm:involution}
    The map $\varphi$ is a sign-reversing involution on words with allowable content, preserves tied inversions (i.e.\ if $i < j$, then $a_i \geq a_j \iff \varphi(a)_i \geq \varphi(a)_j$), and its unique fixed point is $1 2 \cdots k$, which has positive sign.
\end{theorem}

We can lift $\varphi$ to $\FL(\pi)$ by applying it to $w^\ast(\pi, \overline{w})$, using the alphabet $\overline{\Z}_+$. The fixed point of $\varphi$ is the unique fall-labeling of $\pi$ such that $w^\ast(\pi, \overline{w}) = \overline{1} \, \overline{2} \cdots \overline{k}$.

\begin{proposition}
    \label{prop:dinv-involution}
    For $\pi \in \LRP(m+k,n+k)_{\ast k}$, and for $\overline{w} \in \FL(\pi)$, we have
    \[ \falldinv(\pi, \overline{w}) = \falldinv(\pi, \varphi(\overline{w})). \]
\end{proposition}

\begin{proof}
    By definition, this is equivalent to saying that $\tdinv(\psi^{-1}(\pi, \overline{w})) = \tdinv(\psi^{-1}(\pi, \varphi(\overline{w})))$. By the same argument as \cite{GillespieGorskyGriffin2025}*{Lemma~4.6}, the result follows.
\end{proof}

\section{Computing the dinv}
\label{sec:dinv}

This section is dedicated to the proof of the following theorem.

\begin{theorem}
    \label{thm:bijection-dinv}
    Let $\pi \in \LRP(m+k,n+k)_{\ast k}$, and let $\overline{w} \in \FL(\pi)$ be the fall-labeling such that $w^\ast(\pi, \overline{w}) = \overline{1} \, \overline{2} \cdots \overline{k}$. Let $\tilde{\pi} = \psi^{-1}(\pi, \overline{w})$. Then
    \[ \dinv(\pi) = \dinv(\tilde\pi). \]
\end{theorem}

Before proceeding with the proof, let us make some remarks and introduce some notation. First, notice that if $k=0$, then $\psi$ is the identity and \Cref{thm:bijection-dinv} holds, so let us assume $k > 0$.

To lighten the notation, we denote sets referring to the path $\pi$ (and to $\overline{w}$) with plain letters, and sets referring to $\tilde{\pi}$ with a tilde. For example, $C_+ \coloneqq C_+(\pi)$, $\tilde{C}_- \coloneqq C_-(\tilde{\pi})$, and so on.
Using the ENS representation for both paths, we can think of $\tilde\pi$ as the path consisting of the steps $\horizontalsteps \sqcup \verticalsteps \sqcup \decoratedsteps \sqcup \bigsteps \sqcup \southsteps$ and of $\pi$ as the subpath consisting of the steps $\horizontalsteps \sqcup \verticalsteps \sqcup \decoratedsteps$ (see \Cref{fig:psi}).

Recall (cf.~\Cref{def:cdinv_alt}) that \[ \cdinv(\tilde \pi) \coloneqq \# \tilde{C}_+ - \# \tilde{C}_- + \# \tilde{C}^\ast + \# \tilde{B}. \] Since decorated steps of (the ENS representation of) $\tilde\pi$ do not contribute to $\cdinv(\tilde \pi)$, we have $\tilde{C}^\ast = \varnothing$;
since $k>0$, we have $m < n + km$ and therefore $\tilde{C}_+ = \varnothing$.
Let us partition $\tilde{C}_-$ and $\tilde B$ based on the contributions of $\verticalsteps$ (steps with small labels) and $\bigsteps$ (steps with big labels) as follows:
    \begin{align*}
        \tilde{C}_{\text{s}} & \coloneqq \left\{ (i, j) \in \verticalsteps \times \horizontalsteps \mid j < i \text{ and } v_j \leq v_i < v_j + \textstyle\frac{n}{m} + k - 1 \right\} \\
        \tilde{C}_{\text{b}} & \coloneqq \left\{ (i, j) \in \bigsteps \times \horizontalsteps \mid j < i \text{ and } v_j \leq v_i < v_j + \textstyle\frac{n}{m} + k - 1 \right\} \\
        \tilde{B}_{\text{s}} & \coloneqq \left\{ i \in \verticalsteps \mid v_i < 0 \right\} \\
        \tilde{B}_{\text{b}} & \coloneqq \left\{ i \in \bigsteps \mid v_i < 0 \right\}.
    \end{align*}

\begin{remark}
    In the previous expressions for $\tilde{C}_{\text{s}}$ and $\tilde{C}_{\text{b}}$, the $+\,k$ term appears because the slope of the main diagonal of~$\tilde\pi$ is
    \[ \textstyle \frac{n+km}{m} = \frac{n}{m} + k. \]
    In the ENS representation, this means that we are counting the pairs $(i, j) \in \verticalsteps \times \horizontalsteps$ (resp.\ $\bigsteps \times \horizontalsteps$) such that both endpoints of step $i$ diagonally project onto step $j$ or onto any of the $k$ South steps immediately preceding it (right-most endpoint included, left-most endpoint excluded).
    See \Cref{fig:cdinv_example} for an example.
\end{remark}

\begin{figure}
    \centering
    \begin{tikzpicture}[scale=.6]
        \draw[step=1.0, gridcolor, thin] (0,0) grid (5,6);

        \begin{scope}
            \clip (0,0) rectangle (5,2);
            \draw[diagonalcolor, thin]
            (0, 0) -- (5, 2);
        \end{scope}

        \begin{scope}
            \clip (1,1) rectangle (4,4);
            \draw[gray, thin, dashed]
            (1, 1) -- (6, 3)
            (1, 2) -- (6, 4);
        \end{scope}

        \begin{scope}[standardcolor, very thick]
            \draw (-0.05, 0) -- (-0.05,1);
            \draw[sharp <-, sharp angle = 45] (0.05, 1) -- (0.95, 1) -- (0.95, 3);
            \draw[sharp <-, sharp angle = 45] (1.05, 2) -- (1.95, 2) -- (1.95, 3);
            \draw[sharp <-sharp >, sharp angle = 45] (2.05, 3) -- (2.95, 3);
            \draw[sharp <-, sharp angle = 45] (3.05, 1) -- (3.95, 1) -- (3.95, 2);
            \draw[sharp <-, sharp angle = 45] (4.05, 2) -- (5,2);
        \end{scope}

        \begin{scope}[bigcolor, very thick]
            \draw (-0.05,1) -- (-0.05, 4) -- (0, 4);
            \draw (0.95, 3) -- (0.95, 5) -- (1, 5);
            \draw (1.95, 3) -- (1.95, 6) -- (2, 6);
            \draw[sharp <-, sharp angle = 45] (2.95, 3) -- (2.95, 4) -- (3, 4);
            \draw (3.95, 2) -- (3.95, 5) -- (4, 5);
        \end{scope}
        
        \begin{scope}[southcolor, very thick]
            \draw[-sharp >, sharp angle = 45] (0, 4) -- (0.05, 4) -- (0.05, 1);
            \draw (1, 5) -- (1.05, 5) -- (1.05, 3);
            \draw[-sharp >, sharp angle = 45] (2, 6) -- (2.05, 6) -- (2.05, 3);
            \draw (3, 4) -- (3.05, 4) -- (3.05, 3);
            \draw[-sharp >, sharp angle = 45] (4, 5) -- (4.05, 5) -- (4.05, 2);
        \end{scope}

        \begin{scope}[decoratedcolor, very thick]
            \draw[-sharp >, sharp angle = 45] (1.05, 3) -- (1.05, 2);
            \draw[-sharp >, sharp angle = 45] (3.05, 3) -- (3.05, 1);
        \end{scope}

        \begin{scope}[black, very thick]
            \draw (0.95, 1) -- (0.95, 2) (4.05, 2) -- (5, 2);            
        \end{scope}
        
        \begin{scope}[decoratedcolor]
            \draw (1.25,2.5) node {$\ast$};
            \draw (3.25,2.5) node {$\ast$};
            \draw (3.25,1.5) node {$\ast$};
        \end{scope}

        \begin{scope}[every node/.style = {font=\footnotesize}]

            \draw[black]                
                (0.75, 1.5) node {$i$}
                (4.5, 2.3) node {$j$}
                ;
        \end{scope}
    \end{tikzpicture}
    \caption{
        A pair of steps $(i, j) \in \tilde{C}_{\text{s}}$.
        The endpoints of $i \in \verticalsteps$ both diagonally project onto one of the three South steps preceding $j\in \horizontalsteps$.
    }
    \label{fig:cdinv_example}
\end{figure}

Partition also the contributions to $\falldinv(\pi, \overline{w})$ based on the attack relations as follows:
    \begin{align*}
        F_{\text{s}}^\swarrow & \coloneqq \left\{ (i, j) \in \verticalsteps \times \bigsteps \mid j < i \text{ and } i \rightarrow j \right\} \\
        F_{\text{b}}^\swarrow & \coloneqq \left\{ (i, j) \in \bigsteps \times \bigsteps \mid j < i, \; w_i < w_j, \text{ and } i \rightarrow j \right\} \\
        F^\nearrow & \coloneqq \left\{ (i, j) \in (\verticalsteps \sqcup \bigsteps) \times \bigsteps \mid i < j, \; w_i < w_j, \text{ and } i \rightarrow j \right\}.
    \end{align*}

The set $F_{\text{s}}^\swarrow$ (resp.\ $F_{\text{b}}^\swarrow$) consists of all pairs of vertical steps $(i, j)$ such that step $i$ has a small (resp.\ big) label, step $j$ has a big label, and $i$ attacks $j$ from the right.
The set $F^\nearrow$ consists of all pairs of vertical steps $(i, j)$ where $j$ has a big label and $i$ attacks $j$ from the left.

For $s \in [\overline{k}]$, let $\pi_s$ be the subpath obtained from the ENS representation of $\tilde\pi$ by removing all steps $i \in \bigsteps$ with $w_i \geq s$ and the corresponding steps $i' \in \southsteps$ (recall that there is a natural bijection between $\bigsteps$ and $\southsteps$ sending $i \in \bigsteps$ to the step $i' \in \southsteps$ with the same endpoints).
Note that $\pi_{\overline{1}} = \pi$ and $\pi_{\overline{k+1}} = \tilde\pi$.
Two examples are shown in \Cref{fig:dinv-proof,fig:dinv-proof-example2}.

\begin{figure}
    \centering
    \begin{tikzpicture}[scale=.6]
        \begin{scope}[shift={(0, 0)}]
            \draw[step=1.0, gridcolor, thin] (0,0) grid (5,6);
    
            \begin{scope}[standardcolor, very thick]
                \draw (-0.05, 0) -- (-0.05,1);
                \draw[sharp <-, sharp angle = 45] (0.05, 1) -- (0.95, 1) -- (0.95, 3);
                \draw[sharp <-, sharp angle = 45] (1.05, 2) -- (1.95, 2) -- (1.95, 3);
                \draw[sharp <-sharp >, sharp angle = 45] (2.05, 3) -- (2.95, 3);
                \draw[sharp <-, sharp angle = 45] (3.05, 1) -- (3.95, 1) -- (3.95, 2);
                \draw[sharp <-, sharp angle = 45] (4.05, 2) -- (5,2);
            \end{scope}
    
            \begin{scope}[bigcolor, very thick]
                \draw (-0.05,1) -- (-0.05, 3) -- (0, 3);
                \draw (0.95, 3) -- (0.95, 4) -- (1, 4);
                \draw (1.95, 3) -- (1.95, 5) -- (2, 5);
                \draw[sharp <-, sharp angle = 45] (2.95, 3) -- (2.95, 4) -- (3, 4);
                \draw (3.95, 2) -- (3.95, 4) -- (4, 4);

                \begin{scope}[bigcolor!50]
                    \draw (-0.05, 3) -- (-0.05, 4);
                    \draw (0.95, 4) -- (0.95, 5);
                    \draw (1.95, 5) -- (1.95, 6);
                    \draw (3.95, 4) -- (3.95, 5);
                \end{scope}
            \end{scope}

            \begin{scope}[gray]
                \clip (-1.5, 0) rectangle (5, 6);
                \draw (0, 3) -- (-5, 1);
                \draw (1, 4) -- (-4, 2);
                \draw (2, 5) -- (-3, 3);
                \draw (4, 4) -- (-6, 0);
            \end{scope}
            
            \begin{scope}[southcolor, very thick]
                \draw[-sharp >, sharp angle = 45] (0, 3) -- (0.05, 3) -- (0.05, 1);
                \draw (1, 4) -- (1.05, 4) -- (1.05, 3);
                \draw[-sharp >, sharp angle = 45] (2, 5) -- (2.05, 5) -- (2.05, 3);
                \draw (3, 4) -- (3.05, 4) -- (3.05, 3);
                \draw[-sharp >, sharp angle = 45] (4, 4) -- (4.05, 4) -- (4.05, 2);
            \end{scope}
    
            \begin{scope}[decoratedcolor, very thick]
                \draw[-sharp >, sharp angle = 45] (1.05, 3) -- (1.05, 2);
                \draw[-sharp >, sharp angle = 45] (3.05, 3) -- (3.05, 1);
            \end{scope}

            \begin{scope}[decoratedcolor]
                \draw (1.25,2.25) node {$\ast$};
                \draw (3.25,2.25) node {$\ast$};
                \draw (3.25,1.25) node {$\ast$};
            \end{scope}
    
            \begin{scope}[every node/.style = {font=\scriptsize}]
                \draw[bigcolor]
                    (-0.25,1.5) node {$\overline{1}$}
                    (-0.25,2.5) node {$\overline{2}$}
                    (-0.25,3.5) node {$\overline{3}$}
                    (0.75,3.5) node {$\overline{2}$}
                    (0.75,4.5) node {$\overline{3}$}
                    (1.75,3.5) node {$\overline{1}$}
                    (1.75,4.5) node {$\overline{2}$}
                    (1.75,5.5) node {$\overline{3}$}
                    (2.7,3.5) node {$\overline{1}$}
                    (3.75,2.5) node {$\overline{1}$}
                    (3.75,3.5) node {$\overline{2}$}
                    (3.75,4.5) node {$\overline{3}$}
                    ;
    
                \draw[decoratedcolor]
                    (1.25, 2.7) node {$\overline{1}$}
                    (3.25, 1.7) node {$\overline{3}$}
                    (3.25, 2.7) node {$\overline{2}$}
                    ;

            \end{scope}
        \end{scope}

        \begin{scope}[shift={(8, 0)}]
            \draw[step=1.0, gridcolor, thin] (0,0) grid (5,6);
    
            \begin{scope}[standardcolor, very thick]
                \draw (-0.05, 0) -- (-0.05,1);
                \draw[sharp <-, sharp angle = 45] (0.05, 1) -- (0.95, 1) -- (0.95, 3) -- (1, 3);
                \draw[sharp <-, sharp angle = 45] (1.05, 2) -- (1.95, 2) -- (1.95, 3);
                \draw[sharp <-sharp >, sharp angle = 45] (2.05, 3) -- (2.95, 3);
                \draw[sharp <-, sharp angle = 45] (3.05, 1) -- (3.95, 1) -- (3.95, 2);
                \draw[sharp <-, sharp angle = 45] (4.05, 2) -- (5,2);
            \end{scope}
    
            \begin{scope}[bigcolor, very thick]
                \draw (-0.05, 1) -- (-0.05, 2) -- (0, 2);
                \draw (0.95, 3) -- (0.95, 3) -- (1, 3);
                \draw (1.95, 3) -- (1.95, 4) -- (2, 4);
                \draw[sharp <-, sharp angle = 45] (2.95, 3) -- (2.95, 4) -- (3, 4);
                \draw (3.95, 2) -- (3.95, 3) -- (4, 3);

                \begin{scope}[bigcolor!50]
                    \draw (-0.05, 2) -- (-0.05, 3);
                    \draw (0.95, 3) -- (0.95, 4);
                    \draw (1.95, 4) -- (1.95, 5);
                    \draw (3.95, 3) -- (3.95, 4);
                \end{scope}
            \end{scope}

            \begin{scope}[gray]
                \clip (-1.5, 0) rectangle (5, 6);
                \draw (0, 2) -- (-5, 0);
                \draw (1, 3) -- (-4, 1);
                \draw (2, 4) -- (-3, 2);
                \draw (4, 3) -- (-6, -1);
            \end{scope}
            
            \begin{scope}[southcolor, very thick]
                \draw[-sharp >, sharp angle = 45] (0, 2) -- (0.05, 2) -- (0.05, 1);
                \draw (1, 3) -- (1.05, 3) -- (1.05, 3);
                \draw[-sharp >, sharp angle = 45] (2, 4) -- (2.05, 4) -- (2.05, 3);
                \draw (3, 4) -- (3.05, 4) -- (3.05, 3);
                \draw[-sharp >, sharp angle = 45] (4, 3) -- (4.05, 3) -- (4.05, 2);
            \end{scope}
    
            \begin{scope}[decoratedcolor, very thick]
                \draw[-sharp >, sharp angle = 45] (1, 3) -- (1.05, 3) -- (1.05, 2);
                \draw[-sharp >, sharp angle = 45] (3.05, 3) -- (3.05, 1);
            \end{scope}

            \begin{scope}[decoratedcolor]
                \draw (1.25,2.25) node {$\ast$};
                \draw (3.25,2.25) node {$\ast$};
                \draw (3.25,1.25) node {$\ast$};
            \end{scope}
    
            \begin{scope}[every node/.style = {font=\scriptsize}]
                \draw[bigcolor]
                    (-0.25,1.5) node {$\overline{1}$}
                    (-0.25,2.5) node {$\overline{2}$}
                    (0.75,3.5) node {$\overline{2}$}
                    (1.75,3.5) node {$\overline{1}$}
                    (1.75,4.5) node {$\overline{2}$}
                    (2.7,3.5) node {$\overline{1}$}
                    (3.75,2.5) node {$\overline{1}$}
                    (3.75,3.5) node {$\overline{2}$}
                    ;
    
                \draw[decoratedcolor]
                    (1.25, 2.7) node {$\overline{1}$}
                    (3.25, 1.7) node {$\overline{3}$}
                    (3.25, 2.7) node {$\overline{2}$}
                    ;

            \end{scope}
        \end{scope}

        \begin{scope}[shift={(16, 0)}]
            \draw[step=1.0, gridcolor, thin] (0,0) grid (5,6);

            \begin{scope}[standardcolor, very thick]
                \draw (0, 0) -- (0,1) -- (0.95, 1) -- (0.95, 3) -- (1, 3);
                \draw[sharp <-, sharp angle = 45] (1.05, 2) -- (2, 2) -- (2, 3) -- (3, 3);
                \draw[sharp <-, sharp angle = 45] (3, 1) -- (4, 1) -- (4, 2) -- (5, 2);
            \end{scope}
    
            \begin{scope}[bigcolor!50, very thick]
                    \draw (0, 1) -- (0, 2);
                    \draw (2, 3) -- (2, 4);
                    \draw (3, 3) -- (3, 4);
                    \draw (4, 2) -- (4, 3);
            \end{scope}

            \begin{scope}[gray]
                \clip (-1.5, -1) rectangle (5, 6);
                \draw (0, 1) -- (-5, -1);
                \draw (2, 3) -- (-3, 1);
                \draw (3, 3) -- (-2, 1);
                \draw (4, 2) -- (-6, -2);
            \end{scope}
    
            \begin{scope}[decoratedcolor, very thick]
                \draw[-sharp >, sharp angle = 45] (1, 3) -- (1.05, 3) -- (1.05, 2);
                \draw[-sharp >, sharp angle = 45] (3, 3) -- (3, 1);
            \end{scope}

            \begin{scope}[decoratedcolor]
                \draw (1.25,2.25) node {$\ast$};
                \draw (3.25,2.25) node {$\ast$};
                \draw (3.25,1.25) node {$\ast$};
            \end{scope}
    
            \begin{scope}[every node/.style = {font=\scriptsize}]
                \draw[bigcolor]
                    (-0.25,1.5) node {$\overline{1}$}
                    (1.75,3.5) node {$\overline{1}$}
                    (2.7,3.5) node {$\overline{1}$}
                    (3.75,2.5) node {$\overline{1}$}
                    ;
    
                \draw[decoratedcolor]
                    (1.25, 2.7) node {$\overline{1}$}
                    (3.25, 1.7) node {$\overline{3}$}
                    (3.25, 2.7) node {$\overline{2}$}
                    ;

            \end{scope}
        \end{scope}

        \draw (2.5, -1) node {$\pi_{\overline{3}}$};
        \draw (10.5, -1) node {$\pi_{\overline{2}}$};
        \draw (18.5, -1) node {$\pi_{\overline{1}}$};
    \end{tikzpicture}
    \caption{
        The subpaths $\pi_s$ constructed from the paths of \Cref{fig:psi}.
        The half-lines $\ell_i$ and the corresponding steps $i \in \bigsteps$ with $w_i = s$ are also shown.
        We omit the small labels.
    }
    \label{fig:dinv-proof}
\end{figure}

\begin{figure}
    \centering
    \begin{tikzpicture}[scale=.6]
        \begin{scope}[shift={(0, 0)}]
            \draw[step=1.0, gridcolor, thin] (0,-1) grid (5,5);
    
            \begin{scope}[standardcolor, very thick]
                \draw (-0.05, 0) -- (-0.05,1);
                \draw[sharp <-sharp >, sharp angle = 45] (0.05, 1) -- (0.95, 1);
                \draw[sharp <-, sharp angle = 45] (1.05, -1) -- (1.95, -1) -- (1.95, 0);
                \draw[sharp <-, sharp angle = 45] (2.05, 0) -- (2.95, 0) -- (2.95, 1);
                \draw[sharp <-, sharp angle = 45] (3.05, 0) -- (3.95, 0) -- (3.95, 2);
                \draw[sharp <-, sharp angle = 45] (4.05, 2) -- (5,2);
            \end{scope}
    
            \begin{scope}[bigcolor, very thick]
                \draw (-0.05,1) -- (-0.05, 3) -- (0, 3);
                \draw[sharp <-, sharp angle = 45] (0.95, 1) -- (0.95, 2) -- (1, 2);
                \draw (1.95, 0) -- (1.95, 2) -- (2, 2);
                \draw (2.95, 1) -- (2.95, 2) -- (3, 2);
                \draw (3.95, 2) -- (3.95, 4) -- (4, 4);

                \begin{scope}[bigcolor!50]
                    \draw (-0.05, 3) -- (-0.05, 4);
                    \draw (1.95, 2) -- (1.95, 3);
                    \draw (2.95, 2) -- (2.95, 3);
                    \draw (3.95, 4) -- (3.95, 5);
                \end{scope}
            \end{scope}

            \begin{scope}[gray]
                \clip (-1.5, 0) rectangle (5, 6);
                \draw (0, 3) -- (-5, 1);
                \draw (2, 2) -- (-3, 0);
                \draw (3, 2) -- (-2, 0);
                \draw (4, 4) -- (-6, 0);
            \end{scope}
            
            \begin{scope}[southcolor, very thick]
                \draw[-sharp >, sharp angle = 45] (0, 3) -- (0.05, 3) -- (0.05, 1);
                \draw (1, 2) -- (1.05, 2) -- (1.05, 1);
                \draw[-sharp >, sharp angle = 45] (2, 2) -- (2.05, 2) -- (2.05, 0);
                \draw (3, 2) -- (3.05, 2) -- (3.05, 1);
                \draw[-sharp >, sharp angle = 45] (4, 4) -- (4.05, 4) -- (4.05, 2);
            \end{scope}
    
            \begin{scope}[decoratedcolor, very thick]
                \draw[-sharp >, sharp angle = 45] (1.05, 1) -- (1.05, -1);
                \draw[-sharp >, sharp angle = 45] (3.05, 1) -- (3.05, 0);
            \end{scope}

            \begin{scope}[decoratedcolor]
                \draw (1.25,0.25) node {$\ast$};
                \draw (1.25,-0.75) node {$\ast$};
                \draw (3.25,0.25) node {$\ast$};
            \end{scope}
    
            \begin{scope}[every node/.style = {font=\scriptsize}]
                \draw[bigcolor]
                    (-0.25,1.5) node {$\overline{1}$}
                    (-0.25,2.5) node {$\overline{2}$}
                    (-0.25,3.5) node {$\overline{3}$}
                    (0.75,1.5) node {$\overline{2}$}
                    (1.75,0.5) node {$\overline{1}$}
                    (1.75,1.5) node {$\overline{2}$}
                    (1.75,2.5) node {$\overline{3}$}
                    (2.75,1.5) node {$\overline{1}$}
                    (2.75,2.5) node {$\overline{3}$}
                    (3.75,2.5) node {$\overline{1}$}
                    (3.75,3.5) node {$\overline{2}$}
                    (3.75,4.5) node {$\overline{3}$}
                    ;
    
                \draw[decoratedcolor]
                    (1.25, 0.7) node {$\overline{1}$}
                    (1.25, -0.3) node {$\overline{3}$}
                    (3.25, 0.7) node {$\overline{2}$}
                    ;

            \end{scope}
        \end{scope}

        \begin{scope}[shift={(8, 0)}]
            \draw[step=1.0, gridcolor, thin] (0,-1) grid (5,5);
    
            \begin{scope}[standardcolor, very thick]
                \draw (-0.05, 0) -- (-0.05,1);
                \draw[sharp <-, sharp angle = 45] (0.05, 1) -- (0.95, 1);
                \draw[sharp <-, sharp angle = 45] (1.05, -1) -- (1.95, -1) -- (1.95, 0);
                \draw[sharp <-, sharp angle = 45] (2.05, 0) -- (2.95, 0) -- (2.95, 1);
                \draw[sharp <-, sharp angle = 45] (3.05, 0) -- (3.95, 0) -- (3.95, 2);
                \draw[sharp <-, sharp angle = 45] (4.05, 2) -- (5,2);
            \end{scope}
    
            \begin{scope}[bigcolor, very thick]
                \draw (-0.05, 1) -- (-0.05, 2) -- (0, 2);
                \draw (1.95, 0) -- (1.95, 1) -- (2, 1);
                \draw (2.95, 1) -- (2.95, 2) -- (3, 2);
                \draw (3.95, 2) -- (3.95, 3) -- (4, 3);

                \begin{scope}[bigcolor!50]
                    \draw (-0.05, 2) -- (-0.05, 3);
                    \draw (0.95, 1) -- (0.95, 2);
                    \draw (1.95, 1) -- (1.95, 2);
                    \draw (3.95, 3) -- (3.95, 4);
                \end{scope}
            \end{scope}

            \begin{scope}[gray]
                \clip (-1.5, -1) rectangle (5, 6);
                \draw (0, 2) -- (-5, 0);
                \draw (1, 1) -- (-4, -1);
                \draw (2, 1) -- (-3, -1);
                \draw (4, 3) -- (-6, -1);
            \end{scope}
            
            \begin{scope}[southcolor, very thick]
                \draw[-sharp >, sharp angle = 45] (0, 2) -- (0.05, 2) -- (0.05, 1);
                \draw[-sharp >, sharp angle = 45] (2, 1) -- (2.05, 1) -- (2.05, 0);
                \draw (3, 2) -- (3.05, 2) -- (3.05, 1);
                \draw[-sharp >, sharp angle = 45] (4, 3) -- (4.05, 3) -- (4.05, 2);
            \end{scope}
    
            \begin{scope}[decoratedcolor, very thick]
                \draw[-sharp >, sharp angle = 45] (0.95, 1) -- (1.05, 1) -- (1.05, -1);
                \draw[-sharp >, sharp angle = 45] (3.05, 1) -- (3.05, 0);
            \end{scope}

            \begin{scope}[decoratedcolor]
                \draw (1.25,0.25) node {$\ast$};
                \draw (1.25,-0.75) node {$\ast$};
                \draw (3.25,0.25) node {$\ast$};
            \end{scope}
    
            \begin{scope}[every node/.style = {font=\scriptsize}]
                \draw[bigcolor]
                    (-0.25,1.5) node {$\overline{1}$}
                    (-0.25,2.5) node {$\overline{2}$}
                    (0.75,1.5) node {$\overline{2}$}
                    (1.75,0.5) node {$\overline{1}$}
                    (1.75,1.5) node {$\overline{2}$}
                    (2.75,1.5) node {$\overline{1}$}
                    (3.75,2.5) node {$\overline{1}$}
                    (3.75,3.5) node {$\overline{2}$}
                    ;
    
                \draw[decoratedcolor]
                    (1.25, 0.7) node {$\overline{1}$}
                    (1.25, -0.3) node {$\overline{3}$}
                    (3.25, 0.7) node {$\overline{2}$}
                    ;

            \end{scope}
        \end{scope}

        \begin{scope}[shift={(16, 0)}]
            \draw[step=1.0, gridcolor, thin] (0,-1) grid (5,5);
    
            \begin{scope}[standardcolor, very thick]
                \draw[-sharp >, sharp angle = -45] (0, 0) -- (0, 1) -- (1, 1);
                \draw[sharp <-, sharp angle = 45] (1, -1) -- (2, -1) -- (2, 0) -- (2.95, 0) -- (2.95, 1) -- (3, 1);
                \draw[sharp <-, sharp angle = 45] (3.05, 0) -- (4, 0) -- (4, 2) -- (5,2);
            \end{scope}
    
            \begin{scope}[bigcolor, very thick]

                \begin{scope}[bigcolor!50]
                    \draw (0, 1) -- (0, 2);
                    \draw (2, 0) -- (2, 1);
                    \draw (2.95, 1) -- (2.95, 2);
                    \draw (4, 2) -- (4, 3);
                \end{scope}
            \end{scope}

            \begin{scope}[gray]
                \clip (-1.5, -2) rectangle (5, 6);
                \draw (0, 1) -- (-5, -1);
                \draw (2, 0) -- (-3, -2);
                \draw (3, 1) -- (-2, -1);
                \draw (4, 2) -- (-6, -2);
            \end{scope}
            
            \begin{scope}[southcolor, very thick]
            \end{scope}
    
            \begin{scope}[decoratedcolor, very thick]
                \draw[sharp <-sharp >, sharp < angle = 45, sharp > angle = 45] (1, 1) -- (1, -1);
                \draw[-sharp >, sharp angle = 45] (3, 1) -- (3.05, 1) -- (3.05, 0);
            \end{scope}

            \begin{scope}[decoratedcolor]
                \draw (1.25,0.25) node {$\ast$};
                \draw (1.25,-0.75) node {$\ast$};
                \draw (3.25,0.25) node {$\ast$};
            \end{scope}
    
            \begin{scope}[every node/.style = {font=\scriptsize}]
                \draw[bigcolor]
                    (-0.25,1.5) node {$\overline{1}$}
                    (1.75,0.5) node {$\overline{1}$}
                    (2.75,1.5) node {$\overline{1}$}
                    (3.75,2.5) node {$\overline{1}$}
                    ;
    
                \draw[decoratedcolor]
                    (1.25, 0.7) node {$\overline{1}$}
                    (1.25, -0.3) node {$\overline{3}$}
                    (3.25, 0.7) node {$\overline{2}$}
                    ;

            \end{scope}
        \end{scope}

        \draw (2.5, -2) node {$\pi_{\overline{3}}$};
        \draw (10.5, -2) node {$\pi_{\overline{2}}$};
        \draw (18.5, -2) node {$\pi_{\overline{1}}$};
    \end{tikzpicture}
    \caption{
        The subpaths $\pi_s$ and half-lines $\ell_i$ in another example with $m=5$, $n=2$, $k=3$, and $w^\ast(\pi, \overline{w}) = \overline{1} \, \overline{2}\, \overline{3}$.
        In the third picture, two half-lines start below the path, so the set $\tilde{B}_{\text{b}}$ consists of two steps with label $\overline{1}$.
        In this example, part of the path (in the ENS representation) lies below the $x$-axis.
    }
    \label{fig:dinv-proof-example2}
\end{figure}

For $i \in \bigsteps$, denote by $\ell_i$ the half-line parallel to the main diagonal (in the ENS representation) ending at the bottom-most endpoint of step $i$.
We are going to count the number of intersections of these half-lines with certain steps of the path: for the purpose of intersecting $\ell_i$, vertical steps have their top-most endpoint excluded and their bottom-most endpoint included; horizontal steps have their left-most endpoint excluded and their rightmost endpoint included.

Let us now state a few useful counting lemmas.

\begin{lemma}
    \label{lem:tilde_B_big}
    For $i \in \bigsteps$, we have $i \in \tilde B_\text{b}$ if and only if $\ell_i$ starts below the path, i.e.\ it intersects the $y$-axis in a negative $y$-coordinate.
\end{lemma}

\begin{proof}
    The equation of $\ell_i$ is $y = \frac{n}{m} x + v_i$, so it intersects the $y$-axis in $(0, v_i)$. Since $\tilde{B}_{\text{b}} = \left\{ i \in \bigsteps \mid v_i < 0 \right\}$, the claim follows.
    An example where $\tilde{B}_{\text{b}}$ is non-empty is shown in \Cref{fig:dinv-proof-example2}.
\end{proof}

\begin{lemma}
    \label{lem:F_nearrow}
    For $(i,j) \in \left( \verticalsteps \sqcup \bigsteps \right) \times \bigsteps$, we have $(i,j) \in F^\nearrow$ if and only if $i \in \pi_{w_j}$ and $\ell_j$ intersects $i$.
\end{lemma}

\begin{proof}
    By definition of $F^\nearrow$, the pair $(i, j)$ is in $F^\nearrow$ if and only if $i < j, \; w_i < w_j, \text{ and } i \rightarrow j$.
    By definition of $\pi_{w_j}$, we have $w_i < w_j$ if and only if $i \in \pi_{w_j}$. By \Cref{def:attack_relation}, we have $i < j$ and $i \rightarrow j$ if and only if $v_i \leq v_j < v_i + 1$, which is exactly the condition for $\ell_j$ to intersect $i$.
\end{proof}

\begin{figure}
    \centering
    \begin{tikzpicture}[scale=.6]
    
        \begin{scope}[shift={(0, 0)}]
            \begin{scope}[standardcolor, very thick]
                \draw[sharp <-, sharp angle = 45] (0.05, -1) -- (1, -1);
            \end{scope}
    
            \begin{scope}[bigcolor, very thick]
                \draw (-0.05, 1) -- (-0.05, 3) -- (0, 3);
                \draw (2, 2) -- (2, 3);
            \end{scope}

            \begin{scope}[gray]
                \clip (-1.3, 0) rectangle (2, 3);
                \draw (2, 2) -- (-8, -2);
            \end{scope}
            
            \begin{scope}[southcolor, very thick]
                \draw (0, 3) -- (0.05, 3) -- (0.05, 1);
            \end{scope}
    
            \begin{scope}[decoratedcolor, very thick]
                \draw[-sharp >, sharp angle = 45] (0.05, 1) -- (0.05, -1);
            \end{scope}
                
            \begin{scope}[decoratedcolor]
                \draw (0.25, 0.25) node {$\ast$};
                \draw (0.25, -0.75) node {$\ast$};
            \end{scope}
    
            \begin{scope}[every node/.style = {font=\scriptsize}]
                \draw[bigcolor]
                    (-0.255, 1.5) node {$\overline{1}$}
                    (-0.255, 2.5) node {$\overline{3}$}
                    (1.75, 2.5) node {$\overline{4}$}
                    ;
    
                \draw[decoratedcolor]
                    (0.25, 0.7) node {$\overline{2}$}
                    (0.25, -0.3) node {$\overline{5}$}
                    ;

                \draw[black]                
                    (2.25, 2.5) node {$i$}
                    (0.5, -1.3) node {$j$}
                    ;

                \draw[black] (-1, 1.2) node {$\ell_i$};
            \end{scope}

            \draw[
                  decorate,
                  decoration={brace, amplitude=4pt}
            ] (-1.75, 0) -- (-1.75, 3)
              node[midway, xshift=-0.65cm] {\scriptsize $w_i - 1$
              };
                
            \draw[
                  decorate,
                  decoration={brace, amplitude=3.5pt}
            ] (-1.75, -1) -- (-1.75, -0.05)
              node[midway, xshift=-0.4cm] {\scriptsize $r_{ij}'$};
                
        \end{scope}

        \begin{scope}[shift={(9, 0)}]
            \begin{scope}[standardcolor, very thick]
                \draw[sharp <-, sharp angle = 45] (0.05, -1) -- (1, -1);
            \end{scope}
    
            \begin{scope}[bigcolor, very thick]
                \draw (-0.05, 1) -- (-0.05, 2) -- (0, 2);
                \draw (2, 2) -- (2, 3);
            \end{scope}

            \begin{scope}[gray]
                \clip (-1.3, 0) rectangle (2, 3);
                \draw (2, 2) -- (-8, -2);
            \end{scope}
            
            \begin{scope}[southcolor, very thick]
                \draw (0, 2) -- (0.05, 2) -- (0.05, 1);
            \end{scope}
    
            \begin{scope}[decoratedcolor, very thick]
                \draw[-sharp >, sharp angle = 45] (0.05, 1) -- (0.05, -1);
            \end{scope}
                
            \begin{scope}[decoratedcolor]
                \draw (0.25, 0.25) node {$\ast$};
                \draw (0.25, -0.75) node {$\ast$};
            \end{scope}
    
            \begin{scope}[every node/.style = {font=\scriptsize}]
                \draw[bigcolor]
                    (-0.25, 1.5) node {$\overline{1}$}
                    (1.75, 2.5) node {$\overline{2}$}
                    ;
    
                \draw[decoratedcolor]
                    (0.25, 0.7) node {$\overline{2}$}
                    (0.25, -0.3) node {$\overline{5}$}
                    ;

                \draw[black]
                    (-0.25, 0.5) node {$i_*$}
                    (2.25, 2.5) node {$i$}
                    (0.5, -1.3) node {$j$}
                    ;

                \draw[black] (-1, 1.2) node {$\ell_i$};
            \end{scope}

            \draw[
                  decorate,
                  decoration={brace, amplitude=3.5pt}
            ] (-1.75, 1) -- (-1.75, 2)
              node[midway, xshift=-0.65cm] {\scriptsize $w_i - 1$
              };
                
            \draw[
                  decorate,
                  decoration={brace, amplitude=4pt}
            ] (-1.75, -1) -- (-1.75, 0.95)
              node[midway, xshift=-0.4cm] {\scriptsize $r_{ij}'$};
        \end{scope}

    \end{tikzpicture}
    \caption{
        Examples of pairs $(i, j) \in \bigsteps \times \horizontalsteps$ satisfying \eqref{eq:half-line-intersects-before-j} in the proof of \Cref{lem:complicated}.
        The two cases differ by the label $w_i$, which is $\overline{4}$ on the left and $\overline{2}$ on the right; consequently, the path $\pi_{w_i}$ on the right contains fewer steps.
        We have $r_{ij}' = 1$ on the left and $r_{ij}' = 2$ on the right.
    }
    \label{fig:lem-complicated}
\end{figure}

\begin{figure}
    \centering
    \begin{tikzpicture}[scale=.6]
    
        \begin{scope}[shift={(0, 0)}]
            \begin{scope}[standardcolor, very thick]
                \draw[sharp <-, sharp angle = 45] (0.05, -1) -- (1, -1);
            \end{scope}
    
            \begin{scope}[bigcolor, very thick]
                \draw (-0.05, 1) -- (-0.05, 5) -- (0, 5);
                \draw (2, 4) -- (2, 5);
            \end{scope}

            \begin{scope}[gray]
                \clip (-1.4, 0) rectangle (2, 5);
                \draw[dashed] (2, 5) -- (-8, 1);
            \end{scope}
            
            \begin{scope}[southcolor, very thick]
                \draw (0, 5) -- (0.05, 5) -- (0.05, 1);
            \end{scope}
    
            \begin{scope}[decoratedcolor, very thick]
                \draw[-sharp >, sharp angle = 45] (0.05, 1) -- (0.05, -1);
            \end{scope}
                
            \begin{scope}[decoratedcolor]
                \draw (0.25, 0.25) node {$\ast$};
                \draw (0.25, -0.75) node {$\ast$};
            \end{scope}
    
            \begin{scope}[every node/.style = {font=\scriptsize}]
                \draw[bigcolor]
                    (-0.255, 1.5) node {$\overline{1}$}
                    (-0.255, 2.5) node {$\overline{3}$}
                    (-0.255, 3.5) node {$\overline{4}$}
                    (-0.255, 4.5) node {$\overline{6}$}
                    
                    (1.75, 4.5) node {$\overline{4}$}
                    ;
    
                \draw[decoratedcolor]
                    (0.25, 0.7) node {$\overline{2}$}
                    (0.25, -0.3) node {$\overline{5}$}
                    ;

                \draw[black]                
                    (2.25, 4.5) node {$i$}
                    (0.5, -1.3) node {$j$}
                    ;

            \end{scope}

            \draw[
                  decorate,
                  decoration={brace, amplitude=4pt}
            ] (-1.75, 0) -- (-1.75, 2.95)
              node[midway, xshift=-0.65cm] {\scriptsize $w_i - 1$
              };
                
            \draw[
                  decorate,
                  decoration={brace, amplitude=3.5pt}
            ] (-1.75, -1) -- (-1.75, -0.05)
              node[midway, xshift=-0.4cm] {\scriptsize $r_{ij}'$};
                
            \draw[
                  decorate,
                  decoration={brace, amplitude=3.5pt}
            ] (-1.75, 3) -- (-1.75, 4)
              node[midway, xshift=-0.65cm] {\scriptsize $1 - r_{ij}''$};
                
            \draw[
                  decorate,
                  decoration={brace, amplitude=3.5pt}
            ] (-4, -1) -- (-4, 5)
              node[midway, xshift=-0.4cm] {\scriptsize $k$};
                
        \end{scope}

        \begin{scope}[shift={(9, 0)}]
            \begin{scope}[standardcolor, very thick]
                \draw[sharp <-, sharp angle = 45] (0.05, -1) -- (1, -1);
            \end{scope}
    
            \begin{scope}[bigcolor, very thick]
                \draw (-0.05, 1) -- (-0.05, 5) -- (0, 5);
                \draw (2, 3) -- (2, 4);
            \end{scope}

            \begin{scope}[gray]
                \clip (-1.4, 0) rectangle (2, 4);
                \draw[dashed] (2, 4) -- (-8, 0);
            \end{scope}
            
            \begin{scope}[southcolor, very thick]
                \draw (0, 5) -- (0.05, 5) -- (0.05, 1);
            \end{scope}
    
            \begin{scope}[decoratedcolor, very thick]
                \draw[-sharp >, sharp angle = 45] (0.05, 1) -- (0.05, -1);
            \end{scope}
                
            \begin{scope}[decoratedcolor]
                \draw (0.25, 0.25) node {$\ast$};
                \draw (0.25, -0.75) node {$\ast$};
            \end{scope}
    
            \begin{scope}[every node/.style = {font=\scriptsize}]
                \draw[bigcolor]
                    (-0.25, 1.5) node {$\overline{1}$}
                    (-0.25, 2.5) node {$\overline{3}$}
                    (-0.25, 3.5) node {$\overline{4}$}
                    (-0.25, 4.5) node {$\overline{6}$}
                    (1.75, 3.5) node {$\overline{2}$}
                    ;
    
                \draw[decoratedcolor]
                    (0.25, 0.7) node {$\overline{2}$}
                    (0.25, -0.3) node {$\overline{5}$}
                    ;

                \draw[black]
                    (-0.25, 0.5) node {$i_*$}
                    (2.25, 3.5) node {$i$}
                    (0.5, -1.3) node {$j$}
                    ;

            \end{scope}

            \draw[
                  decorate,
                  decoration={brace, amplitude=3.5pt}
            ] (-1.75, 1) -- (-1.75, 2)
              node[midway, xshift=-0.65cm] {\scriptsize $w_i - 1$
              };
                
            \draw[
                  decorate,
                  decoration={brace, amplitude=4pt}
            ] (-1.75, -1) -- (-1.75, 0.95)
              node[midway, xshift=-0.4cm] {\scriptsize $r_{ij}'$};

            \draw[
                  decorate,
                  decoration={brace, amplitude=3.5pt}
            ] (-4, -1) -- (-4, 5)
              node[midway, xshift=-0.4cm] {\scriptsize $k$};
        \end{scope}

    \end{tikzpicture}
    \caption{
        Examples of pairs $(i, j) \in \bigsteps \times \horizontalsteps$ satisfying \eqref{eq:step-i-attacks-big-steps-above-j} in the proof of \Cref{lem:complicated}.
        The dashed gray line indicates the attacking relation: step $i$ attacks a step $i'$ above the left-most endpoint of $j$.
        Note that $r_{ij}'' = 0$ on the left ($i_*$ is not defined) and $r_{ij}'' = 1$ on the right ($i_*$ is defined).
    }
    \label{fig:lem-complicated2}
\end{figure}

\begin{lemma}
    \label{lem:complicated}
    The expression $\#C^* + \# \tilde C_\text{b} - \# F_\text{b}^\swarrow$ counts the intersections of the half-lines $\ell_i$ with the steps in $\horizontalsteps \sqcup \decoratedsteps \sqcup \southsteps$ belonging to $\pi_{w_i}$, summed over all $i\in\bigsteps$.
\end{lemma}

\begin{proof} %

    The proof strategy is to interpret all numbers in the statement as counting pairs of steps $i \in \bigsteps$ and $j \in \horizontalsteps$, with $j < i$, satisfying certain conditions.
    For any such pair $(i, j)$, denote by $r_{ij}'$ the number of decorated steps $i' \in \decoratedsteps$ immediately preceding step $j$ with $\overline{w}(i') \geq w_i$.
    If the step with fall-label equal to $w_i$ is among these, denote it by $i_*$ (recall that there is exactly one such step in the entire path $\pi$).
    Note that $v_{i_*} = v_j + \frac{n}{m} + r_{ij}' - 1$ by definition of $i_*$.
    Let $r_{ij}'' = 1$ if $i_*$ is defined and $0$ otherwise.

    For a fixed step $i \in \bigsteps$, let us count the intersections of the half-line $\ell_i$ with the steps in $\horizontalsteps \sqcup \decoratedsteps \sqcup \southsteps$ belonging to $\pi_{w_i}$.
    This is equal to the number of steps $j \in \horizontalsteps$ with $j < i$ such that $\ell_i$ intersects step~$j$ or one of the steps in $\decoratedsteps \sqcup \southsteps$ immediately before step~$j$ belonging to $\pi_{w_i}$ (see \Cref{fig:lem-complicated}).
    For a given $j \in \horizontalsteps$, this happens if and only if
    \begin{equation}
        v_j \leq v_i < v_j + \textstyle\frac{n}{m} +  r_{ij}' + w_i - 1.
        \label{eq:half-line-intersects-before-j}
    \end{equation}
    Here we treat $w_i$ as a positive integer, even though technically $w_i \in \overline{\Z}_+$.

    Next, let us rewrite the term $\# F_\text{b}^\swarrow$.
    For $i \in \bigsteps$ and $j \in \horizontalsteps$ with $j < i$, there is (exactly) one pair $(i, i') \in F_\text{b}^\swarrow$, with $i'$ appearing above the left-most endpoint of $j$, if and only if
    \begin{equation}
        v_j + \textstyle\frac{n}{m} + r_{ij}' + w_i - 1 - r_{ij}'' \leq v_i < v_j + \textstyle\frac{n}{m} + k - 1;
        \label{eq:step-i-attacks-big-steps-above-j}
    \end{equation}
    this is exemplified in \Cref{fig:lem-complicated2}.

    \begin{figure}
        \centering
        \begin{tikzpicture}[scale=.6]
            \begin{scope}[shift={(0, 0)}]
                \begin{scope}[standardcolor, very thick]
                    \draw[sharp <-, sharp angle = 45] (0.05, -1) -- (1, -1);
                \end{scope}
        
                \begin{scope}[bigcolor, very thick]
                    \draw (-0.05, 1) -- (-0.05, 5) -- (0, 5);
                    \draw (3, 3) -- (3, 4);
                \end{scope}
    
                \begin{scope}[gray]
                    \clip (-1.4, 0) rectangle (3, 4);
                    \draw[dashed] (3, 2) -- (-7, -2);
                \end{scope}
                
                \begin{scope}[southcolor, very thick]
                    \draw (0, 5) -- (0.05, 5) -- (0.05, 1);
                \end{scope}
        
                \begin{scope}[decoratedcolor, very thick]
                    \draw[-sharp >, sharp angle = 45] (0.05, 1) -- (0.05, -1);
                    \draw (3, 2) -- (3, 1);
                \end{scope}
                    
                \begin{scope}[decoratedcolor]
                    \draw (0.25, 0.25) node {$\ast$};
                    \draw (0.25, -0.75) node {$\ast$};

                    \draw (3.25, 1.25) node {$\ast$};
                \end{scope}
        
                \begin{scope}[every node/.style = {font=\scriptsize}]
                    \draw[bigcolor]
                        (-0.25, 1.5) node {$\overline{1}$}
                        (-0.25, 2.5) node {$\overline{3}$}
                        (-0.25, 3.5) node {$\overline{4}$}
                        (-0.25, 4.5) node {$\overline{6}$}
                        (2.75, 3.5) node {$\overline{2}$}
                        ;
        
                    \draw[decoratedcolor]
                        (0.25, 0.7) node {$\overline{2}$}
                        (0.25, -0.3) node {$\overline{5}$}

                        (3.25, 1.7) node {$\overline{3}$}
                        ;
    
                    \draw[black]
                        (-0.25, 0.5) node {$i_*$}
                        (3.25, 3.5) node {$i$}
                        (0.5, -1.3) node {$j$}
                        (3.8, 1.5) node {$j'$}
                        ;
    
                \end{scope}
    
                \draw[
                      decorate,
                      decoration={brace, amplitude=3.5pt, mirror}
                ] (4.5, 2) -- (4.5, 3)
                  node[midway, xshift=0.65cm] {\scriptsize $w_i - 1$
                  };

            \end{scope}
    
        \end{tikzpicture}
        \caption{
            Example of a pair $(i, j) \in \bigsteps \times \horizontalsteps$ satisfying \eqref{eq:istar-jprime-in-c-star} in the proof of \Cref{lem:complicated}.
            The dashed gray line indicates that, in order to have $(i_*, j') \in C^*$, the top-most endpoint of $j'$ must diagonally project onto $i_*$.
        }
        \label{fig:lem-complicated3}
    \end{figure}

    Fix a pair $(i, j)$ such that $r_{ij}'' = 1$ (so $i_*$ is defined).
    Look at contributions to $C^*$ from pairs $(i_*, j')$ where $j'>i$ has at least one endpoint below~$i$.
    If $(i_*, j') \in C^*$, the decorated steps immediately before $j'$ have a label $< w_i$, whereas all decorated steps immediately after $j'$ (including $j'$) have a label $> w_i$, by construction of $\overline w$.
    Therefore, step~$j'$ is uniquely determined by $v_i = v_{j'}^+ + w_i - 1$, and by definition of $C^*$, we have $(i_*, j') \in C^*$ if and only if $v_{i_*} \leq v_{j'}^+ < v_{i_*} + 1$.
    This condition can be rewritten as
    \begin{equation}
        v_j + \textstyle\frac{n}{m} + r_{ij}' + w_i - 1 - r_{ij}'' \leq v_i < v_j + \textstyle\frac{n}{m} + r_{ij}' + w_i - 1
        \label{eq:istar-jprime-in-c-star}
    \end{equation}
    (note that this condition is vacuous if $r_{ij}'' = 0$).
    See \Cref{fig:lem-complicated3}.

    The term $\# \tilde C_\text{b}$ is already defined in terms of pairs $(i, j) \in \bigsteps \times \horizontalsteps$ with $j < i$; such a pair contributes to $\# \tilde C_\text{b}$ if and only if
    \begin{equation}
        v_j \leq v_i < v_j + \textstyle\frac{n}{m} + k - 1.
        \label{eq:tilde-c-b}
    \end{equation}

    We conclude by observing that a pair $(i, j)$ satisfies \eqref{eq:half-line-intersects-before-j} or \eqref{eq:step-i-attacks-big-steps-above-j} if and only if \eqref{eq:tilde-c-b} holds; it satisfies both \eqref{eq:half-line-intersects-before-j} and \eqref{eq:step-i-attacks-big-steps-above-j} if and only if \eqref{eq:istar-jprime-in-c-star} holds.
\end{proof}

\begin{lemma}
    \label{lem:C_star_plus_C_small_equals_C_minus_plus_F_small}
    We have $\# C_+ + \# \tilde{C}_{\text{s}} = \# C_- + \# F_{\text{s}}^\swarrow$.
\end{lemma}

\begin{proof}
    In $F_{\text{s}}^\swarrow$, we can group the steps with a big label by the horizontal step following them, obtaining the set
    \[ F_{\text{s}}' = \left\{
        (i, j) \in \verticalsteps \times \horizontalsteps \mid j < i \text{ and } v_j + \textstyle\frac{n}{m} + r_j - 1 \leq v_i < v_j + \textstyle\frac{n}{m} + k - 1
    \right\}, \]
    and clearly $\# F_\text{s}' = \# F_\text{s}^\swarrow$.
    Now, $C_+ \sqcup \tilde{C}_{\text{s}} = C_- \sqcup F_{\text{s}}'$ (the sets are disjoint) as they are both equal to
    \[  \left\{ (i, j) \in \verticalsteps \times \horizontalsteps \mid j < i \text{ and } v_j + \min\left(0, \, \textstyle\frac{n}{m} + r_j - 1 \right) \leq v_i < v_j + \textstyle\frac{n}{m} + k - 1 \right\}. \]
    The result follows.
\end{proof}

We can now proceed with the proof of \Cref{thm:bijection-dinv}.

\begin{proof}[Proof of \Cref{thm:bijection-dinv}]
    By definition, we have $\dinv(\pi) = \tdinv(\pi) + \cdinv(\pi)$ and $\dinv(\tilde\pi) = \tdinv(\tilde\pi) + \cdinv(\tilde\pi)$.
    By \Cref{rmk:falldinv}, we have $\tdinv(\tilde\pi) = \tdinv(\pi) + \falldinv(\pi, \overline{w})$, so it remains to prove the following:
    \[ \cdinv(\pi) =  \cdinv(\tilde\pi) + \falldinv(\pi, \overline{w}). \]
    Using the previously shown partitionings, we can rewrite the statement as
    \begin{equation}
        \# C_+ - \# C_- + \# C^* + \# B = - \# \tilde{C}_\text{s} - \# \tilde{C}_\text{b} + \# \tilde B_\text{s} + \# \tilde B_\text{b} + \# F_{\text{s}}^\swarrow + \# F_{\text{b}}^\swarrow + \# F^\nearrow. \label{eq:main}
    \end{equation}
    By construction, we have $B = \tilde{B}_{\text{s}}$. By \Cref{lem:C_star_plus_C_small_equals_C_minus_plus_F_small}, we have $\# C_+ - \# C_- = \# F_{\text{s}}^\swarrow - \# \tilde{C}_{\text{s}}$. Therefore, we can cancel out the corresponding summands from \eqref{eq:main} and reduce the statement to
    \[
        (\# C^\ast + \# \tilde{C}_{\text{b}} - \# F_{\text{b}}^\swarrow) - \# F^\nearrow = \# \tilde B_\text{b}.
    \]
    By \Cref{lem:F_nearrow}, the term $\# F^\nearrow$ counts the intersections of the half-lines $\ell_i$ with the steps in $\verticalsteps \sqcup \bigsteps$ belonging to~$\pi_{w_i}$; by \Cref{lem:complicated}, the term $\#C^* + \# \tilde C_\text{b} - \# F_\text{b}^\swarrow$ counts the intersections of the half-lines $\ell_i$ with the steps in $\horizontalsteps \sqcup \decoratedsteps \sqcup \southsteps$ belonging to $\pi_{w_i}$.

    Now, fix $i \in \bigsteps$.
    If we traverse $\ell_i$ in the top-right direction, every time $\ell_i$ intersects a step in $\horizontalsteps \sqcup \decoratedsteps \sqcup \southsteps$ it goes from weakly below the path $\pi_{w_i}$ to strictly above the path, and every time it intersects a step in $\verticalsteps \sqcup \bigsteps$ it does the opposite.
    Since $\ell_i$ ends weakly above the path, these numbers are the same if $\ell_i$ starts weakly above the path, and they are off by one if $\ell_i$ starts strictly below the path, in which case $i \in \tilde B_\text{b}$.
    Summing the contributions over all steps $i \in \bigsteps$, the result follows.
\end{proof}

\section{Proof of \texorpdfstring{\Cref{thm:main}}{the main theorem}}
\label{sec:proof-main}

It is now time to put together the pieces and conclude the proof of \Cref{thm:main}. First, we recall the statement of the \emph{rational shuffle theorem}.

\begin{theorem}[Rational shuffle theorem {\cite{Mellit2021Rational}}]
    \label{thm:rational}
    For $m, n \in \mathbb{N}$, we have 
    \[ e_{m,n} = \sum_{\pi \in \LRD(m,n)} q^{\dinv(\pi)} t^{\area(\pi)} x^\pi. \] 
\end{theorem}

Applying the skewing operator $s_{(m-1)^k}^\perp$ to the combinatorial interpretation of $e_{m,n+km}$ given by \Cref{thm:rational}, and using \Cref{prop:schur-expansion}, we obtain
\[ s_{(m-1)^k}^\perp e_{m,n+km} = \sum_{\alpha \text{ allowable}} (-1)^{\sgn(\alpha)} \sum_{\tilde\pi \in \LRD(m,n+km)^{\tilde\alpha}} q^{\dinv(\tilde\pi)} t^{\area(\tilde\pi)} x^{\tilde\pi}. \]
Combining this with \Cref{prop:shape-bijection}, we get
\[ s_{(m-1)^k}^\perp e_{m,n+km} = \sum_{\pi \in \LRD(m+k,n+k)_{\ast k}} \sum_{\overline{w} \in \FL(\pi)} (-1)^{\sgn(\overline{w})} q^{\dinv(\psi^{-1}(\pi, \overline{w}))} t^{\area(\psi^{-1}(\pi, \overline{w}))} x^\pi. \]
By \Cref{prop:dinv-involution} and \Cref{thm:bijection-dinv}, this equality becomes
\[ s_{(m-1)^k}^\perp e_{m,n+km} = \sum_{\pi \in \LRD(m+k,n+k)_{\ast k}} q^{\dinv(\pi)} t^{\area(\pi)} x^\pi, \]
which is exactly \Cref{thm:main}, as desired.

Let us recall the statement of the \emph{rectangular shuffle conjecture}.

\begin{conjecture}[{\cite{IraciPagariaPaoliniVandenWyngaerd2023}*{Conjecture~4.2}}]
    For $m, n \in \mathbb{N}$ and $d = \gcd(m,n)$, we have \[ \frac{[m]_q}{[d]_q} p_{m,n} = \sum_{\pi \in \LRP(m,n)} q^{\dinv(\pi)} t^{\area(\pi)} x^\pi. \]
\end{conjecture}

Using the fact that $\gcd(m,n+km) = \gcd(m,n)$, the previous argument also proves \Cref{thm:rectangular}. In particular, due to \cite{IraciPagariaPaoliniVandenWyngaerd2023}*{Theorem~6.1}, the equality of \Cref{thm:rectangular} holds when $d=1$.

\section{Links with the existing literature}
\label{sec:links}

The objects and the symmetric functions introduced in this paper reduce to several other families of objects and symmetric functions in the literature, for suitable choices of $m$, $n$, and $k$.
In this section, we give a non-exhaustive overview of results that can be deduced from the previous sections, together with several new interesting conjectures.

\subsection{\texorpdfstring{$D_\alpha$}{D alpha} operators and \texorpdfstring{$e$}{e}-positivity}

Let $\alpha \vDash n$ be defined by $\alpha_i \coloneqq \left\lceil \frac{ni}{m} \right\rceil - \left\lceil \frac{n(i-1)}{m} \right\rceil$, and let $x^\alpha \coloneqq \prod_{i} x_i^{\alpha_i}$. Recall that, in the notation of \cite{BlasiakHaimanMorsePunSeelinger2023ShuffleAnyLine} (cf.\ \cites{BlasiakHaimanMorsePunSeelinger2023ProofExtendedDelta,GillespieGorskyGriffin2025}), we have the identity \[ e_{m,n} = D_\alpha(1) = \omega \left( H_{q,t}^m \left( \frac{x^\alpha}{\prod_i (1-qtx_i/x_{i+1})} \right)_{\textnormal{pol}} \right). \]

Then, by \cite{GillespieGorskyGriffin2025}*{Lemma~3.12}, we can deduce the following identity. 

\begin{proposition}
    For $m,n,k \in \mathbb{N}$ with $m > 0$, and $\alpha$ as above, we have
    \[ s_{(m-1)^k}^{\perp} e_{m,n+km} = \omega \left(H_{q,t}^m \left( \frac{x^\alpha h_k(x_1, \dots, x_m)}{\prod_i (1-qtx_i/x_{i+1})} \right)_{\textnormal{pol}} \right). \]
\end{proposition}

Since the operator $ H_{q,t}^m \left( \cdot \right)_{\textnormal{pol}}$ is linear, we can split the sum into single summands. The monomials in $h_k(x_1 \cdots x_m)$ are naturally indexed by weak compositions $\beta \vDash k$ of length $\ell(\beta) \leq m$, where $\beta_i$ is the exponent of $x_i$.

\begin{definition}
    Let $(\pi, \decoratedsteps, w) \in \LRP(m+k,n+k)_{\ast k}$. We define its \emph{fall-composition} $\beta(\pi) \vDash k$ as $\beta(\pi) \coloneqq (r_i)_{i \in \horizontalsteps}$, that is, $\beta_i$ is the number of decorated falls immediately preceding the $i\th$ nondecorated horizontal step of $\pi$.
\end{definition}

We can naturally split the generating function of $\LRD(m+k,n+k)_{\ast k}$ according to $\beta(\pi)$. Ideally, we would like to match each summand to a monomial in $h_k(x_1, \dots, x_m)$. This is not possible in general because for a generic $\beta$ the symmetric function $D_{\alpha+\beta}(1)$ is not monomial-positive. However, if we set $q=1$, we get the following.

\begin{theorem}
    \label{thm:q=1}
    For $m,n,k \in \mathbb{N}$ with $m > 0$, $\alpha$ as above, and $\beta \vDash k$ of length $\ell(\beta) \leq m$, we have
    \[
        \left. D_{\alpha+\beta}(1) \right\rvert_{q=1} = \sum_{\substack{\pi \in \LRD(m+k,n+k)_{\ast k} \\ \beta(\pi) = \beta}} t^{\area(\pi)} x^\pi.
    \]
\end{theorem}

\Cref{thm:q=1} is essentially a reformulation of \cite{BlasiakHaimanMorsePunSeelinger2023ShuffleAnyLine}*{Equation~(148)}, which states \[ D_\gamma(1) = \sum_{\lambda} t^{a(\lambda)} e_{b(\lambda)}, \] where $\lambda$ is any path above the one prescribed by $\gamma$ (that is, the one with $\gamma_i$ vertical steps on the vertical line $x = i - 1$), $a(\lambda)$ is the area enclosed between the two paths, and $b(\lambda)$ is the composition such that $b(\lambda)_i$ is the number of vertical steps of $\lambda$ on the vertical line $x = i - 1$. In order to derive \Cref{thm:q=1}, it is sufficient to contract the columns containing decorated falls. The area stays the same, and the obtained path is one of the paths above the shape prescribed by $\gamma$. This correspondence is trivially bijective once one prescribes the positions of the decorated falls.

However, this is interesting because this formula refines the fall Delta theorem and our rational extension in a way that was not, to the authors' awareness, known before. It would be interesting to find a formula for such summands that refines $\Delta'_{e_n} e_{n+k} = \Theta_{e_k} \nabla e_n$ in some way.
Indeed, \Cref{thm:q=1} has the same combinatorial flavor as \cite{HicksRomero2018Polyominoes}*{Theorem~1.1} (cf.\ \cite{IraciRomero2022DeltaTheta}*{Theorem~8.6}), as they both split a positive generating function into non-positive pieces that become positive and combinatorial when specializing one variable to $1$.

\subsection{Theta operators}
\label{sec:theta-operators}

Another attempt at finding a decorated extension of the rational shuffle theorem is \cite{IraciPagariaPaoliniVandenWyngaerd2023}*{Conjecture~4.4}, which states the following.

\begin{conjecture}
    For any $m, n \in \mathbb{N}$, we have \[ \left. \Theta_{e_k} e_{m,n} \right\rvert_{q=1} = \sum_{\pi \in \LRD(m+k,n+k)^{\ast k}} t^{\area(\pi)} x^\pi, \]
    where $\LRD(m+k,n+k)^{\ast k}$ denotes the set of rise-decorated labeled rectangular Dyck paths of size $(m+k) \times (n+k)$.
\end{conjecture}

The main difference from \Cref{thm:main} is that the decorations are on the rises rather than the falls. The problem of finding a $\dinv$ statistic for rise-decorated paths remains open.
While the distinction is almost irrelevant when $m=n$, when $m \neq n$ we get really different sets of objects. However, the distinction disappears again when looking at non-labeled paths, or even Schr\"oder paths, as in both cases the reflection with respect to the antidiagonal yields a bijection between the sets of objects.

It is known that, in order to extract Schr\"oder paths, one has to take the scalar product with $h_d e_{n+k-d}$, where $d$ is the number of diagonal steps. In terms of labelings, taking such a scalar product corresponds to picking $d$ peaks of the path labeled with a maximal label, and then labeling the rest of the path in increasing order with respect to the distance from the diagonal (taken vertically or horizontally depending on the set of objects involved). In particular, this suggests the following symmetric function identity, backed up by computer experiments.

\begin{conjecture}
    For $m, n, k, d \in \mathbb{N}$ with $m > 0$ and $d \leq \min\{m+k, n+k\}$, we have \[ \< e_{m,n+km}, s_{(m-1)^k} h_d e_{n+k-d} \> = \< \Theta_{e_k} e_{n,m}, h_d e_{m+k-d} \>. \]
\end{conjecture}

In principle, this identity could be exploited to find a $\dinv$ statistic on rise-decorated rectangular Dyck paths. While our first attempts have not been successful, this warrants further investigation.

\subsection{The Delta square conjecture}

In \cite{DAdderioIraciVandenWyngaerd2019DeltaSquareConjecture}, the authors conjecture a decorated version of the square paths theorem, in terms of rise-decorated square paths. We refer to \cites{DAdderioIraciVandenWyngaerd2019DeltaSquareConjecture,IraciVandenWyngaerd2021ValleySquare} for all missing definitions. The statement is the following.

\begin{conjecture}[{\cite{DAdderioIraciVandenWyngaerd2019DeltaSquareConjecture}*{Conjecture~3.12}}]
    For $n, k \in \mathbb{N}$, we have
    \[ \frac{[n]_t}{[n+k]_t} \Delta_{e_{n}} \omega(p_{n+k}) = \sum_{\pi \in \LRP(n+k,n+k)^{\ast k}} q^{\dinv(\pi)} t^{\area(\pi)} x^\pi, \]
    where $\LRP(n+k,n+k)^{\ast k}$ denotes the set of rise-decorated labeled square paths of size $n+k$.
\end{conjecture}

Notice the extra factor $[n]_t/[n+k]_t$ in the expression. In \cite{IraciVandenWyngaerd2021ValleySquare}, the authors partly address this issue. Via the identity \[ \frac{[n]_t}{[n+k]_t} \Delta_{e_{n}} \omega(p_{n+k}) = \frac{[n+k]_q}{[n]_q} \Theta_{e_k} \nabla \omega(p_n) \] and the $q,t$-reversed version \[ \frac{[n]_q}{[n+k]_q} \Delta_{e_{n}} \omega(p_{n+k}) = \frac{[n+k]_t}{[n]_t} \Theta_{e_k} \nabla \omega(p_n), \] the authors formulate a variant of the conjecture in terms of valley-decorated square paths.

A contractible valley of a labeled rectangular path is a vertical step that is preceded by either two horizontal steps, or one horizontal step such that the label of the vertical step immediately before it is strictly smaller than the label of the vertical step immediately after it. These decorations first appeared in the \emph{valley version} of the Delta conjecture \cite{HaglundRemmelWilson2018DeltaConjecture}. The authors give the following two conjectures.

\begin{conjecture}[{\cite{IraciVandenWyngaerd2021ValleySquare}*{Conjecture~5}}]
    For $n, k \in \mathbb{N}$, we have
    \[ \frac{[n]_t}{[n+k]_t} \Theta_{e_k} \nabla \omega(p_n) = \sum_{\pi \in \LRP(n+k, n+k)^{\bullet k}} q^{\dinv(\pi)} t^{\area(\pi)} x^\pi, \]
    where $\LRP(n+k, n+k)^{\bullet k}$ denotes the set of valley-decorated labeled square paths of size $n+k$.
\end{conjecture}

\begin{conjecture}[{\cite{IraciVandenWyngaerd2021ValleySquare}*{Conjecture~8}}]
    \label{conj:valley-square}
    For $n, k \in \mathbb{N}$, we have
    \[ \Theta_{e_k} \nabla \omega(p_n) = \sum_{\pi \in \LRP'(n+k, n+k)^{\bullet k}} q^{\dinv(\pi)} t^{\area(\pi)} x^\pi, \]
    where $\LRP'(n+k, n+k)^{\bullet k}$ is the subset of $\LRP(n+k, n+k)^{\bullet k}$ where the bottom-most vertical step of the path among those lying on the base diagonal is not decorated.
\end{conjecture}

Notice that \Cref{conj:valley-square} does not have any multiplicative factor. Indeed, even when $m = n$, the combinatorial symmetry of rises and falls does not hold for generic rectangular paths, because of the condition of ending with a horizontal step: since they might start with a horizontal step, the symmetry with respect to the antidiagonal does not preserve the set.

This suggests that decorating rises might not be the correct approach, and we should look at falls instead. Indeed, computer experiments suggest the following.

\begin{conjecture}[Delta square conjecture, fall version]
    \label{conj:fall-square}
    For $n, k \in \mathbb{N}$ with $n > 0$, we have
    \[ \Theta_{e_k} \nabla \omega(p_n) = \sum_{\pi \in \LRP(n+k, n+k)_{\ast k}} q^{\dinv(\pi)} t^{\area(\pi)} x^\pi. \]    
\end{conjecture}

Since we already have a formula for the right-hand side of \Cref{conj:fall-square}, we claim the following.

\begin{conjecture}
    For $n, k \in \mathbb{N}$, we have
    \[ \Theta_{e_k} \nabla \omega(p_n) = s_{(n-1)^k}^\perp p_{n,n(k+1)}. \]    
\end{conjecture}

\section*{Acknowledgements}

We thank Matteo Migliorini for an insight that led us to the ENS representation.
We also thank the anonymous referees for their useful suggestions.
Finally, we acknowledge the support of MUR grant PRIN 2022A7L229 ``Algebraic and topological combinatorics'' CUP-J53D23003660006.
The authors are members of the research group INdAM--GNSAGA.

\bibliographystyle{amsalpha}
\bibliography{references}

\end{document}